\documentclass[11pt]{article}\usepackage{fullpage,amsmath,amsfonts,amsthm,graphicx,amssymb,textcomp,enumerate,multicol,rotating}
\usepackage{hyperref}

\newtheorem{theorem}{Theorem}[section]
\newtheorem{corollary}[theorem]{Corollary}
\newtheorem{lemma}[theorem]{Lemma}
\newtheorem{proposition}[theorem]{Proposition}

\theoremstyle{definition}
\newtheorem{definition}[theorem]{Definition}

\theoremstyle{remark}
\newtheorem{example}[theorem]{Example}

\theoremstyle{remark}
\newtheorem{remark}[theorem]{Remark}

\theoremstyle{remark}

\newcommand{\C}{\mathbb{C}}

\newcommand{\N}{\mathbb{N}}

\newcommand{\Z}{\mathbf{Z}}  
\newcommand{\T}{\mathbf{T}}
\newcommand{\I}{\mathbf{I}}
\newcommand{\D}{\mathbf{D}}
\newcommand{\E}{\mathbf{E}}
\newcommand{\F}{\mathbf{F}}
\newcommand{\G}{\mathbf{G}}
\newcommand{\X}{\mathcal{X}}
\newcommand{\la}{\left\langle}
\newcommand{\ra}{\right\rangle}

\newcommand{\vres}{\includegraphics[scale=0.18]{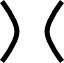}}
\newcommand{\hres}{\includegraphics[scale=0.18]{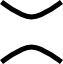}}
\newcommand{\TLei}{\includegraphics[scale=0.35]{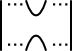}}
\newcommand{\crossing}{\includegraphics[scale=0.18]{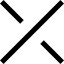}}
\newcommand{\ncrossing}{\includegraphics[scale=0.18]{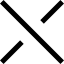}\text{ }}

\newcommand{\vsimeq}{\begin{sideways}\begin{sideways}\begin{sideways}$\simeq$\end{sideways}\end{sideways}\end{sideways}}
\newcommand{\twoheaddownarrow}{\begin{sideways}$\twoheadleftarrow$\end{sideways}}

\title{A Colored Khovanov Homotopy Type And Its Tail For B-Adequate Links}
\author{Michael Willis \\
Department of Mathematics, University of Virginia\\
\href{mailto:msw3ka@virginia.edu}{\texttt{msw3ka@virginia.edu}}}

\begin{document}

\maketitle

\begin{abstract}
We define a Khovanov homotopy type for $\mathfrak{sl}_2(\C)$ colored links and quantum spin networks and derive some of its basic properties.  In the case of $n$-colored B-adequate links, we show a stabilization of the homotopy types as the coloring $n\rightarrow\infty$, generalizing the tail behavior of the colored Jones polynomial.  Finally, we also provide an alternative, simpler stabilization in the case of the colored unknot.
\end{abstract}

\section{Introduction}
In \cite{Khov} Mikhail Khovanov introduced the Khovanov homology $Kh^{i,j}(L)$ of a knot or link $L$, the homology of a bigraded chain complex $KC^{i,j}(L)$ with graded Euler characteristic equal to the Jones polynomial of $L$.  In \cite{LS} Robert Lipshitz and Sucharit Sarkar defined the Khovanov homotopy type of $L$, a wedge sum of spectra $\X^j(L)$ whose reduced cohomology groups satisfy $\tilde{H}^i(\X^j(L))\cong Kh^{i,j}(L)$.  It is then natural to ask what types of structural results about Khovanov homology extend to the Khovanov homotopy type.

In \cite{MW} the author proved that, for the torus links $T(n,m)$, the Khovanov homotopy types $\X(T(n,m))$ stabilize as $m\rightarrow\infty$, allowing a well-defined notion of a limiting Khovanov homotopy type $\X(T(n,\infty))$.  Due to Lev Rozansky's arguments in \cite{Roz}, this result could be interpreted as defining a colored Khovanov homotopy type for the $n$-colored unknot.  (Here `colored' refers to assigning an irreducible $\mathfrak{sl}_2(\C)$ representation, as determined by its dimension $n\in\N$, to each component of the link.)  In this paper, we prove the following two extensions of \cite{MW}.

\begin{theorem}
\label{Colored X exists}
There exists a stable colored Khovanov homotopy type for any $\mathfrak{sl}_2(\C)$ colored link.  Its reduced cohomology is isomorphic to the colored Khovanov homology defined in \cite{CK}, \cite{Roz2}.
\end{theorem}

\begin{theorem}
\label{Spin Network X exists}
There exists a stable Khovanov homotopy type for any $\mathfrak{sl}_2(\C)$ quantum spin network.  Its reduced cohomology is isomorphic to the homology of the categorified spin networks defined in \cite{CK}.
\end{theorem}

Both the colored Khovanov homology and the categorified quantum spin networks mentioned in these theorems are defined using the categorified Jones-Wenzl projectors.  Thus theorems \ref{Colored X exists} and \ref{Spin Network X exists} will be viewed as special cases of a slightly more general theorem that can be stated as follows.

\begin{theorem}
\label{Projectors X exists}
For any link diagram $D$ involving a finite number of Jones-Wenzl projectors, there exists a stable Khovanov homotopy type $\X(D)$ with reduced cohomology isomorphic to the homology defined using the categorified Jones-Wenzl projectors as in \cite{Roz} and \cite{CK}.
\end{theorem}

An example of the type of diagram in the statement of Theorem \ref{Projectors X exists} is provided in Figure \ref{GenWenzEx}.  Notice that the Jones-Wenzl projectors themselves, and their categorifications, are defined using tangles.  The Khovanov homotopy type has not yet been defined for tangles, and so Theorem \ref{Projectors X exists} requires that the projectors involved are closed in some way to form a link diagram.  Nevertheless we will also prove several properties of such homotopy types, such as being `killed by turnbacks', that the projectors and their categorifications satisfy.

\begin{figure}[h]
\centering
\vspace{.1in}\includegraphics[scale=.5]{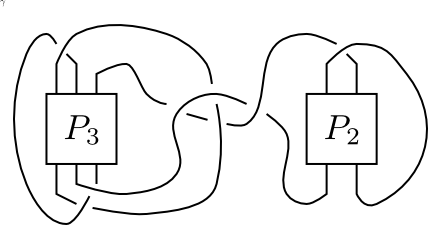}
\caption{An example diagram for which Theorem \ref{Projectors X exists} defines a Khovanov homotopy type}
\vspace{.1in}
\label{GenWenzEx}
\end{figure}

The proof of Theorem \ref{Projectors X exists} is a generalization of the proof in \cite{MW} for the torus links.  With \cite{Roz} in mind, we replace the projectors with torus braids, seeking a stabilization of the homotopy types as the number of twists in each such braid goes to infinity.  The strategy is similar to that in \cite{MW}, but requires some new bounds and estimates that account for the presence of crossings away from the twisting, as well as the (possibly changing) orientations of the strands being twisted.

\begin{remark}
\label{alt paper}
Theorem \ref{Colored X exists} was proved independently by Andrew Lobb, Patrick Orson, and Dirk Schuetz in \cite{LOS} which appeared on the Arxiv while this manuscript was in preparation.  The authors further remark in that paper that their methods could be used to prove a statement similar to Theorem \ref{Projectors X exists}.
\end{remark}

With a well-defined colored Khovanov homotopy type in hand, we follow the strategy of \cite{Roz2} to prove:

\begin{theorem}
\label{B-adequate tail}
The Khovanov homotopy types of $n$-colored B-adequate links stabilize as $n\rightarrow\infty$.
\end{theorem}
For a more detailed statement and an illustration of the stabilization, see Theorem \ref{B-adequate tail better} and Tables \ref{Colored table} and \ref{Hopf link Colored table} in Section 5.  Theorem \ref{B-adequate tail} gives us the stable tail behavior for the Khovanov homotopy types of colored B-adequate links, matching the behavior of the colored homology and colored Jones polynomials.  This theorem is a lifting of Theorem 2.2 in \cite{Roz2} to the stable homotopy category.  The proof will be based on two main ideas.  First we verify that all of the isomorphisms constructed in \cite{Roz2} between colored Khovanov homology groups lift to maps between the corresponding homotopy types.  Second we ensure that the homological range of isomorphism for these maps (which depends on $n$) can be translated into a range of $q$-degrees for which all non-zero homology is isomorphic via these same maps.  This will allow Whitehead's theorem to guarantee that the maps are stable homotopy equivalences.

Finally, we will also provide a more direct argument for the tail of the Khovanov homotopy type of the colored unknot; that is, for $\X(T(n,\infty))$ as $n\rightarrow\infty$.

\begin{theorem}
\label{n to infty}
In the case of the unknot, the $n$-colored Khovanov homotopy types $\X(T(n,\infty))$ defined in \cite{MW} stabilize as $n\rightarrow\infty$ and the stable limit $\X(T(\infty,\infty)):=\bigvee_{j\in\left(2\N\cup 0\right)}\X^j(T(\infty,\infty))$ satisfies
\[\X^j(T(\infty,\infty))\simeq\X^0(T(j,\infty))\simeq\X^{(j-1)^2}(T(j,j-1)) \hspace{.3in}\text{for }j>0\]
\[\X^{0}(T(\infty,\infty))\simeq\X^{-1}(T(1,\infty))\simeq S^0\]
where $S^0$ denotes the standard sphere spectrum.
\end{theorem}

For a more detailed visual representation of the statement of Theorem \ref{n to infty}, see Table \ref{Unknot colored table} in Section 6.  The proof of Theorem \ref{n to infty} will use much simpler stable homotopy equivalences than the maps used to prove Theorem \ref{B-adequate tail}, and will also provide a sharper bound on the coloring $n$ needed for stabilization in a given $q$-degree.

This paper is arranged as follows.  In Section 2 we review the necessary background on Khovanov homology, the Khovanov homotopy type, and the categorified Jones-Wenzl projectors as constructed in \cite{Roz}.  We also set our grading conventions for Khovanov homology used throughout the paper.  In Section 3 we build the Khovanov homotopy type for arbitrary diagrams involving Jones-Wenzl projectors, proving Theorem \ref{Projectors X exists}, and then derive some simple properties for these homotopy types similar to those satisfied by the projectors themselves.  Section 4 contains a short description of quantum spin networks and colored links, allowing quick proofs of Theorems \ref{Colored X exists} and \ref{Spin Network X exists} via Theorem \ref{Projectors X exists}.  Section 5 is devoted to proving Theorem \ref{B-adequate tail}.  Finally, Section 6 contains the proof of Theorem \ref{n to infty}, with one proof from this section placed in the Appendix.

The author would like to thank: the referee of his prior paper \cite{MW} for bringing up the question of allowing $n\rightarrow\infty$ as in Theorem \ref{n to infty}; Matt Hogancamp for calling attention to the properties of linking a 1-colored unknot with a colored link; and his advisor Slava Krushkal for his continued support and advice while preparing this paper.

\section{Background}

\subsection{Our Grading Conventions for Khovanov Homology}
For the original definition of the Khovanov homology of a link, see \cite{Khov}.  We quickly summarize here the main points.  Any crossing (\crossing) in a link diagram can be resolved in one of two ways: with a 0-resolution (\vres) or with a 1-resolution (\hres).  The Khovanov chain complex $KC^{i,j}(L)$ of a link diagram $L$ is a bigraded chain complex built out of a cube of resolutions of the diagram $L$.  The generators of $KC^{i,j}(L)$ correspond to assignments of $v_+$ or $v_-$ to each circle in any given resolution.  There are several different conventions in the literature for the precise meaning of the two gradings $i$ and $j$.  In this paper, following \cite{LS} and \cite{BN}, we shall let $i$ refer to the homological grading and $j$ refer to the $q$-grading, which we define by
\begin{equation}
\label{homdeg}
\deg_h(\cdot) := \#(\text{1-resolutions}) - n^-
\end{equation}
\begin{equation}
\label{qdeg}
\deg_q(\cdot) := \#(\text{1-resolutions}) + (\#(v_+) - \#(v_-)) + (n^+-2n^-)
\end{equation}
where $n^+$ and $n^-$ denote the number of positive and negative crossings respectively in the diagram for $L$.  Under these grading conventions, the Khovanov differential increases $\deg_h$ by one and respects $\deg_q$, allowing $KC^{i,j}(L)$ to split as a direct sum over $q$-degree.  The resulting homology groups are then bigraded invariants of the link, with no shifts necessary for any Reidemeister moves on the diagram used.  In what follows, the $q$-grading normalization shift $n^+-2n^-$ will often be denoted by $N$.

\subsection{The Khovanov Homotopy Type}
Given a link $L$ in $S^3$, we shall let $\X(L)=\bigvee_{j\in\mathbb{Z}} \X^j(L)$ denote the Khovanov homotopy type of the link $L$.  For the full description of this invariant, see \cite{LS}.  We summarize the important points about $\X(L)$ here:
\begin{itemize}
\item $\X(L)$ is the suspension spectrum of a CW complex.
\item $\tilde{H}^i(\X^j(L))=Kh^{i,j}(L)$, the Khovanov homology of the link $L$ (see Equations \ref{homdeg} and \ref{qdeg} for our grading conventions).
\item Each $\X^j(L)$ is constructed combinatorially using the Khovanov chain complex $KC^j(L)$ in $q$-degree $j$, together with a choice of ``Ladybug Matching'' that uses the diagram for $L$ (see section 5.4 in \cite{LS}).  Note that since $KC^j(L)$ is nontrivial for only finitely many $q$-degrees $j$, the wedge sum above is actually finite.
\item Each $\X^j(L)$ is an invariant of the link $L$.  That is to say, the stable homotopy type of $\X^j(L)$ does not depend on the diagram used to portray $L$, nor on the various choices that are made during the construction.
\item Non-trivial Steenrod square operations on $\tilde{H}^i(\X^j(L))=Kh^{i,j}(L)$ can serve to differentiate links with isomorphic Khovanov homology \cite{LS2} and also give rise to slice genus bounds \cite{LS3}.  One corollary of the work in \cite{MW} gives the existence of non-trivial $Sq^2$ for infinitely many 3-strand torus links.  See further calculations in \cite{LOS}.
\end{itemize}

The most important property of $\X(L)$ for our purposes comes from the following `Collapsing Lemma', a slight generalization of that appearing in section 2.2 of \cite{MW}.  Fixing $j\in\mathbb{Z}$, we consider the Khovanov chain complex $KC(L)$ represented as the mapping cone of a chain map:
\begin{equation}
\label{Lemma3.32format}
KC^{j+N_L}(L) = \big(KC^{j+N_{L''}}(L'') \longrightarrow KC^{j-1+N_{L'}}(L')\big)
\end{equation}
where $L'$ and $L''$ are the links resulting from taking the 1-resolution and 0-resolution, respectively, of a single crossing in the diagram for $L$.  The superscripts stand for $q$-gradings, with $N_L$ denoting the $q$-degree normalization shift $n^+-2n^-$ in the link diagram $L$, and similarly for $N_{L'}$ and $N_{L''}$.  There is a corresponding cofibration sequence of homotopy types (see Theorem 2 in \cite{LS}):
\begin{equation}
\label{Lemma3.32onX}
\Sigma^a\X^{j+N_{L''}}(L'')\hookrightarrow\X^{j+N_L}(L)\twoheadrightarrow\Sigma^b\X^{j-1+N_{L'}}(L')
\end{equation}
where the $\Sigma$ stands for suspensions allowing for shifts in homological degree, with  $a=n^-_L-n^-_{L''}$ and $b=n^-_L-n^-_{L'}-1$, the differences in the count of negative crossings $n^-$ for the various diagrams (the extra $-1$ for $L'$ takes into account the loss of a 1-resolution from the point of view of $L'$).  See Equations \ref{homdeg} and \ref{qdeg} above to clarify the grading shifts.

\begin{lemma}
\label{Lemma3.32}
With $KC^{j+N_L}(L) = \big(KC^{j+N_{L''}}(L'') \longrightarrow KC^{j-1+N_{L'}}(L')\big)$ as above, we have:
\begin{itemize}
\item If $KC^{j-1+N_{L'}}(L')$ is acyclic, then the induced inclusion $\Sigma^a\X^{j+N_{L''}}(L'')\hookrightarrow\X^{j+N_L}(L)$ is a stable homotopy equivalence.
\item If $KC^{j+N_{L''}}(L'')$ is acyclic, then the induced surjection $\X^{j+N_L}(L)\twoheadrightarrow\Sigma^b\X^{j-1+N_{L'}}(L')$ is a stable homotopy equivalence.
\end{itemize}
\end{lemma}
\begin{proof}
See the brief sketch in \cite{MW}, which describes the first case for positive crossings.  Both cases are special cases of Lemma 3.32 in \cite{LS}, presented as in Theorem 2 from the same paper.
\end{proof}

Lemma \ref{Lemma3.32} says that we can resolve crossings in a diagram one a time, and if one resolution of a crossing results in a diagram with acyclic chain complex in the specified $q$-degree, this entire part of the full chain complex can be collapsed and we are left with the chain complex using only the other resolution (up to some potential suspensions).  Just as in \cite{MW}, we will want to make repeated use of this idea here.

\subsection{A Categorified Jones-Wenzl Projector}

In the Temperley-Lieb algebra $TL_n$ on $n$ strands over coefficient field $\C(q)$, we have a special idempotent element $P_n$ characterized by the following axioms:
\begin{enumerate}[I.]
\item $P_n\cdot e_i = e_i\cdot P_n = 0$ for any of the standard multiplicative generators $e_i=\TLei\in TL_n$.  This is often described by stating that $P_n$ is ``killed by turnbacks''.
\item The coefficient of the $n$-strand identity tangle in the expression for $P_n$ is 1.
\end{enumerate}
(For the original definition of the $P_n$, see \cite{Wenzl}; for an account of the Temperley-Lieb algebra, the $P_n$, and some of their uses in 3-manifold theory, see \cite{KL}.)

In \cite{Roz} Lev Rozansky provided a categorification for any $P_n$ via an infinite torus braid.  If we let $\sigma_1,\dots,\sigma_{n-1}$ denote the standard generators of the braid group $B_n$, we introduce the following notation for full twists on $n$ strands:
\begin{equation}
\label{T notation}
\T^k_{n}:=(\sigma_1\sigma_2\cdots\sigma_{n-1})^{nk}.
\end{equation}
After giving a well-defined notion for a stable limit of chain complexes, Rozansky proved the following theorem.

\begin{theorem}[Rozansky]
\label{Roz Inf Twist}
The Khovanov chain complexes associated to the braids $KC(\T_n^k)$ stabilize up to chain homotopy as $k\rightarrow\pm\infty$.  The limiting complex $KC(\T_n^{\pm\infty})$ satisfies the following properties:
\begin{enumerate}[I.]
\item Adding a turnback onto the top or bottom of $KC(\T_n^{\pm\infty})$ causes the entire complex to be chain homotopic to a trivial complex.
\item The resulting complex can be viewed as a mapping cone of a map from (for $+\infty$) or to (for $-\infty$) the 1-term complex of the identity tangle, where the other terms involve only non-identity tangles in non-zero homological degrees.
\end{enumerate}
\end{theorem}
\begin{proof}
See \cite{Roz}, and also section 1.6 in \cite{Roz2}.
\end{proof}

This theorem means that, to obtain a chain complex categorifying the Jones-Wenzl projectors up through a given homological degree, it is enough to replace any $P_n$ in a diagram with a copy of $\T_n^{\pm k}$ for large $k$ (we shall often refer to this as a `finite-twist approximation').  The exact size of $k$ needed depends on the homological degree we are interested in.  The graded Euler characteristic of this complex stabilizes as $k\rightarrow\infty$ to give a power series representation of the rational terms appearing in the usual formulas for the $P_n$.  Positive (right-handed) twisting gives a power series in $q$, while negative (left-handed) twisting gives a power series in $q^{-1}$.

\begin{remark}
\label{CKs and Khov cat}
At around the same time, Ben Cooper and Slava Krushkal independently constructed a categorification of the Jones-Wenzl projectors in \cite{CK}.  We are unsure if it is possible to lift their construction to the Khovanov homotopy type; see \cite{LOS} for some further remarks.  Also, in \cite{Khov2}, Mikhail Khovanov introduced separate categorifications for the colored Jones polynomial using renormalizations to eliminate the denominators present in the terms of the Jones-Wenzl projectors.  Our approach here aims to recover Rozansky's version outlined above rather than these alternative categorifications (although the categorified projectors in \cite{CK} are chain homotopic to those produced in \cite{Roz}; see section 3 of \cite{CK}).
\end{remark}

The first goal of this paper is to properly lift Theorem \ref{Roz Inf Twist} to a similar statement about Khovanov homotopy types.  However, Theorem \ref{Roz Inf Twist} is a statement about complexes of tangles.  As mentioned in the introduction, the Khovanov homotopy type has not yet been defined for tangles in general, and only exists for links.  This is the reason for the slightly indirect phrasing of Theorem \ref{Projectors X exists}, which can be viewed as a statement about having Khovanov homotopy types for Jones-Wenzl projectors that are closed up in any fashion in $S^3$.  We turn now to the proof of Theorem \ref{Projectors X exists}.

\section{A Khovanov Homotopy Type For Diagrams Involving Jones-Wenzl Projectors}

\subsection{Basic Notations and a Key Counting Lemma}

We begin with some general notation for use throughout this section.

\begin{itemize}
\item $n\in\N$ will always denote a number of strands for various purposes (typically for a given torus braid or, later, for an $n$-strand cabling of a link diagram).
\item Boldface capital letters will refer to braids and/or tangles within a diagram.
\item $\I_n$ will denote the identity braid on $n$ strands.
\item $\T_n^k$ will denote a torus braid on $n$ strands with $k$ full right-handed (positive) twists (see Equation (\ref{T notation})).
\item $\T_n^{-k}$ will denote such a torus braid with $k$ full left-handed (negative) twists.
\item $\Z$ will often be used to denote an arbitrary tangle.
\item We will use the inner product notation $\la \Z_1,\Z_2 \ra$ to indicate connecting two tangles top to top and bottom to bottom.  This notation is meant to imitate the inner product in the Temperley-Lieb algebra.  See Figure \ref{TZcap}.
\item $\Z^{\cap i}$ will be used to indicate that the $i^{\text{th}}$ and $(i+1)^{\text{st}}$ strands at the top of the tangle $\Z$ are being capped off.  Similarly, $\Z_{\cup i}$ will indicate capping off the $i$ and $i+1$ strands at the bottom.  See Figure \ref{TZcap}.
\end{itemize}

\begin{figure}[h]
\centering
\vspace{.1in}\includegraphics[scale=.2]{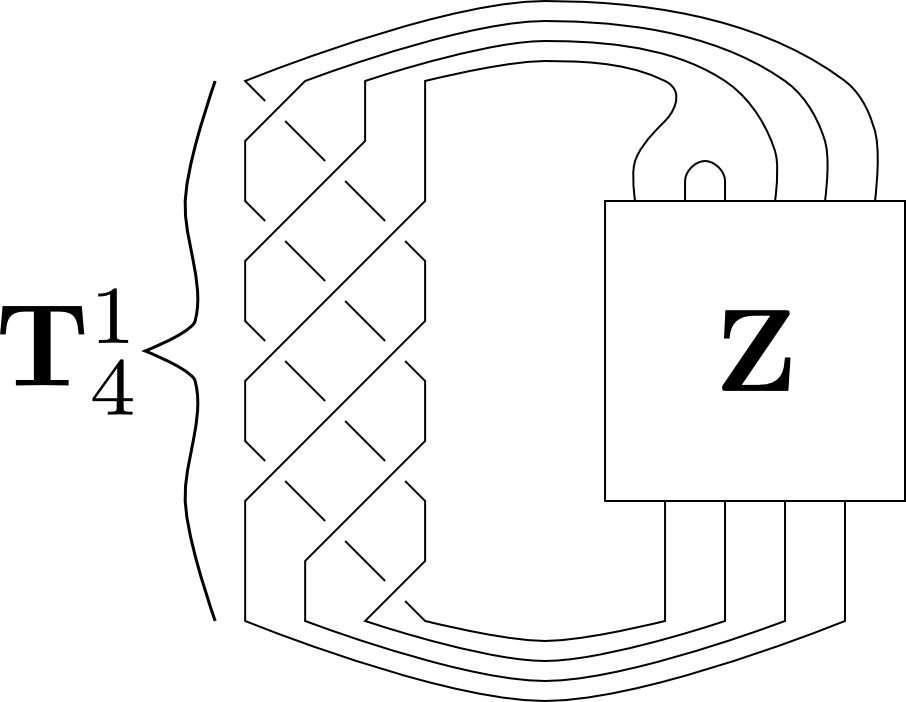}
\caption{The diagram $\la \T_4^1 , \Z^{\cap 2}\ra$.  $\T_4^1$ indicates the full right-handed twist on 4 strands, and $\Z$ is some fixed $(6,4)$-tangle.  The $\cap 2$ indicates a cap on the 2nd and 3rd strands above $\Z$.}
\vspace{.1in}
\label{TZcap}
\end{figure}

In many link diagrams in this paper, a single copy of $\T_n^{\pm k}$ will be singled out for consideration, allowing the diagram to be viewed as $\la \T_n^{\pm k} , \Z\ra$ for some tangle $\Z$ (similar to Figure \ref{TZcap}, but without the cap).  In these situations, we will also view the normalization shifts of Equations (\ref{homdeg}) and (\ref{qdeg}) as split into contributions based on the $\T_n$ and the $\Z$ as in the following definition.

\begin{definition}
\label{tau N def}
In a link diagram $L$ viewed as $L = \la \T_n^{\pm k} , \Z \ra$ as above, the symbol $\tau$ will be used to denote the $q$-normalization shift $n^+-2n^-$ counting only crossings within \emph{one full twist} of the $n$ strands (that is, within $\T_n^{\pm 1}$).  Similarly, the symbol $\eta$ will be used to denote the homological normalization shift $n^-$ counting only crossings within \emph{one full twist}.  The symbol $N_\Z$ will be used to denote the $q$-normalization shift $n^+-2n^-$ counting only crossings within the tangle $\Z$.  More generally, $N_D$ will denote the shift $n^+-2n^-$ counting crossings within a diagram $D$ (whether tangle or otherwise).
\end{definition}

We will have no need for the homological normalization shift $n^-$ counting only crossings in $\Z$.

\begin{remark}
\label{tau N rmk}
Notice that these shifts $\tau$, $\eta$ and $N_\Z$ depend on the orientation of the strands, and allowable orientations are affected by the \emph{full link diagram involved}, not just the piece being counted.  In Figure \ref{TZcap} for example, the value of $\tau$ depends heavily on the tangle $\Z$, despite the fact that it only counts crossings within the $\T_4^1$.  In cases where the tangle $\Z$ is changing, subscripts will be attached to the symbol $\tau$ as necessary to indicate which full diagram is being considered.  Similarly, if there are multiple $\T_{n_i}$ to consider within some single link diagram, the subscript $i$ will be used for the shifts $\tau_i$ and $\eta_i$ to indicate which twist is being considered.
\end{remark}

In order to illustrate these notations, we prove the following very simple observation about full twists that indicates why they are preferable to work with (as opposed to the fractional twists that were sufficient in \cite{MW}).

\begin{lemma}
\label{NZlemma}
For any $(n,n)$-tangle $\Z$, consider the diagram $D(k):=\la \T_n^{k} , \Z \ra$.  Then all of the $D(k)$ are links with the same number of strands, which can be oriented equivalently for all $k$.  Thus $N_{D(k)}=k\tau+N_\Z$ with $N_\Z$ and $\tau$ independent of $k$ (in particular, $N_\Z$ can be determined by the diagram $D(0)=\la \I,\Z \ra$).  Similarly for such a diagram, $\eta$ is also independent of $k$.
\end{lemma}
\begin{proof}
In a full twist, any strand takes the $i^\text{th}$ point at the top to the $i^\text{th}$ point on the bottom, so for the purposes of counting and orienting the strands, this is equivalent to the identity braid $\I$.  The orientations of the strands are all that matters for calculating $N_\Z$, and also for calculating $\tau$ and $\eta$.  Since $\tau$ counts positive and negative crossings for one full twist, $k$ full twists will contribute $k\tau$.
\end{proof}

\begin{remark}
\label{NZlemma neg} The previous observation was written and notated for positive full twists, but it is clear that the exact same argument holds for negative full twists as well.  This will be typical of several of the arguments later in this section.
\end{remark}

We conclude this section with the key counting lemma which is used essentially throughout the paper.  This lemma can be viewed as a generalization of Lemma 3.5 in \cite{MW}, which itself was just a restatement of Marko Sto\v{s}i\'{c}'s Lemma 1 in \cite{Sto}.

\begin{lemma}
\label{Counting lemma}
Fix $n\geq 2$ in $\N$.  Then for any $i\in\{1,2,\dots,n-1\}$ and for any $(n-2,n)$-tangle $\Z$, consider the link diagram $D_\pm = \la (\T_n^{\pm k})^{\cap i} , \Z \ra$.  That is, consider any closure of $\T_n^{\pm k}$ involving at least one turnback at the top.  Then for any chosen orientation of the strands we have:
\begin{itemize}
\item This link diagram is isotopic to $D'_\pm = \la \T_{n-2}^{\pm k} , \Z_{\cup n-i} \ra$
\item Letting $\tau_\pm$ count $n^+-2n^-$ for crossings from $\T_n$ in $D_\pm$ and letting $\tau'_\pm$ count this shift for crossings from $\T_{n-2}$ in  $D'_\pm$, we have
\begin{equation}
\label{pos tau'}
\tau'_+ = \tau_+ +2n
\end{equation}
\begin{equation}
\label{neg tau'}
\tau'_- = \tau_- +2n - 6
\end{equation}
\end{itemize}
\end{lemma}
\begin{proof}
We pull the turnback through the full twists, which corresponds to pulling out two `parallel' strands wrapping around the cylinder defining the torus braid.  As in Lemma \ref{NZlemma}, using full twists ensures that the turnback `exits' the torus braid at the same two points that it entered, which swing around to give the $(n-i)^{\text{th}}$ and $(n-i+1)^{\text{st}}$ points at the bottom of $\Z$.  This leaves us with $n-2$ strands for the torus braid, still with the same amount of twisting.  See Figure \ref{CountingLemmaPic}.  This proves the first point.

\begin{figure}[h]
\centering
\vspace{.1in}\includegraphics[scale=.3]{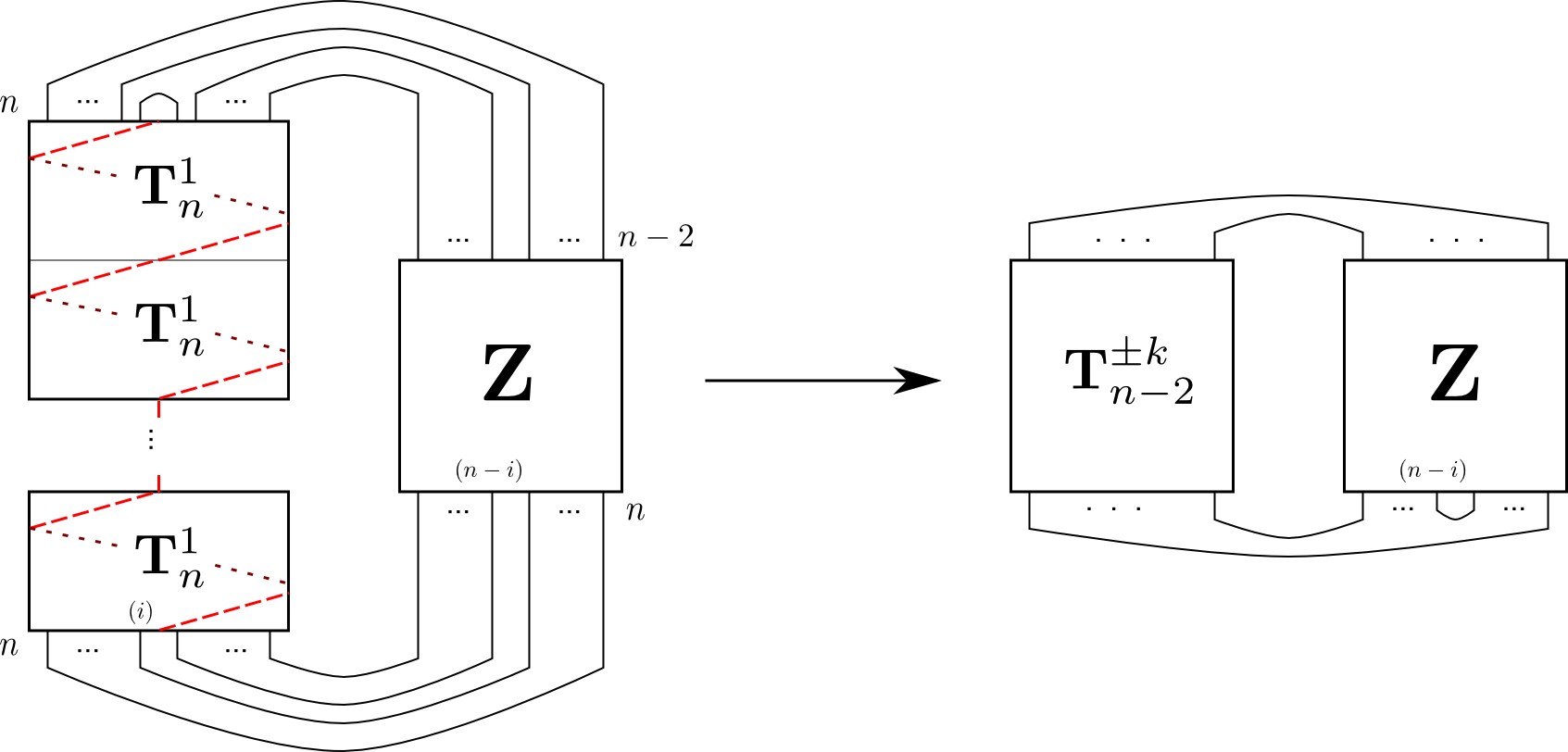}
\caption{The diagram $\la (\T_n^{\pm k})^{\cap i} , \Z \ra$ with the $\T_n^{\pm k}$ drawn as separate $\T_n^1$'s.  The cap is pulled through the twists as shown (the dashed red line would be for $+k$; the opposite direction would be taken for $-k$).  The $n$ and $n-2$ show the number of strands entering and exiting at various points.  The $(i)$ at the bottom of the twisting indicates the $i^\text{th}$ strand counted from the left, and similarly for the $(n-i)$ at the bottom of $\Z$.}
\vspace{.1in}
\label{CountingLemmaPic}
\end{figure}

To prove the second point, we first note that the total number of crossings in a full twist on $n$ strands is $n(n-1)$, while the total number for a full twist on $n-2$ strands is $(n-2)(n-3)$.  This means that when pulling the turnback through, we managed to eliminate $4n-6$ crossings.  One full twist of these two strands corresponds to two Reidemeister 1 moves; the other $4n-8$ eliminations all must have come from Reidemeister 2 moves.  Regardless of the type of twist and the orientation of the strands, all of these Reidemeister 2 moves would have eliminated one positive and one negative crossing each.  The two Reidemeister 1 moves would have eliminated negative crossings from a positive twist, or eliminated positive crossings from a negative twist.  Again, this is independent of the orientation of the strands.  Calculating the effect of these eliminations on the normalization $n^+-2n^-$ gives the result.
\end{proof}

\begin{remark}
There is no difference in having the turnback at the bottom of the $\T_n^{\pm k}$.  The proof makes it clear that it ends up at the top of the $\Z$ in that case.
\end{remark}

\subsection{Proving Theorem \ref{Projectors X exists}}

Let $D$ denote a link diagram involving a finite number of Jones-Wenzl projectors.  More precisely, $D$ is obtained from a link diagram by formally replacing a finite number of identity braids $\I_{n_i}$ with Jones-Wenzl projectors $P_{n_i}$.  (Figure \ref{GenWenzEx} in the introduction provides clarification).  Just as in section 7 of \cite{MW}, we would like to define $\X^j(D)$ as the homotopy colimit of a sequence of homotopy types of finite link diagrams that stabilizes as the twisting in the diagram goes to infinity.  To do this we focus on a single projector at a time.  Towards that end, we combine Lemmas \ref{Lemma3.32} and \ref{NZlemma} to establish the following two sequences.

\begin{proposition}
\label{building seq}
Fix $n\in\N$ and $j\in\mathbb{Z}$.  Let $\Z$ be an arbitrary $(n,n)$-tangle.  Then the maps of Lemma \ref{Lemma3.32} provide the following two sequences (one for right-handed twists, one for left-handed twists):
\begin{equation}
\label{pos seq}
\X^{j+N_\Z}\left( \la \T_n^0 , \Z \ra \right) \hookrightarrow
\Sigma^{-\eta}\X^{j+N_\Z+\tau}\left( \la \T_n^1 , \Z \ra \right) \hookrightarrow \cdots \hookrightarrow
\Sigma^{-k\eta}\X^{j+N_\Z+k\tau}\left( \la \T_n^k , \Z \ra \right) \hookrightarrow \cdots
\end{equation}
\begin{equation}
\label{neg seq}
\X^{j+N_\Z}\left( \la \T_n^0 , \Z \ra \right) \twoheadleftarrow \cdots \twoheadleftarrow
\Sigma^{k(-\eta+n(n-1))}\X^{j+N_\Z+k\tau+kn(n-1)}\left( \la \T_n^{-k} , \Z \ra \right) \twoheadleftarrow \cdots
\end{equation}
where the symbols $\eta$, $\tau$, and $N_\Z$ are as defined in Definition \ref{tau N def}.
\end{proposition}

\begin{proof}
To build the right-handed sequence (\ref{pos seq}), we `start' with the $(k+1)^\text{st}$ term and resolve crossings within one of the full twists one at a time until we reach the $k^\text{th}$ term.  Specifically, we consider the diagram $\la \T_n^{k+1} , \Z \ra = \la \T_n^k , \T_n^1 \cdot \Z \ra$ where we use the product notation to indicate concatenation (see Figure \ref{DandEposseq}).  We number the crossings of the $\T_n^1$ sitting above $\Z$ starting from the `topmost' such crossing.  Then each inclusion in (\ref{pos seq}) will be defined as the composition of $n(n-1)$ inclusions coming from (\ref{Lemma3.32onX}) by resolving these numbered crossings as 0-resolutions in this order (note that the all-zero resolution of $\T_n^1$ is precisely $\I_n$).

We now introduce some notation similar to the notation in section 3 of \cite{MW}.  Let $\D_0:=\T_n^1\cdot\Z$.  Then for each $i=1,2,\dots,n(n-1)$, let $\D_i$ denote the diagram obtained from $\D_{i-1}$ by resolving the $i^\text{th}$ crossing with a 0-resolution, and let $\E_i$ denote the diagram obtained from $\D_{i-1}$ (\emph{not} from $\E_{i-1}$; this will change for the left-handed sequence) by resolving the $i^\text{th}$ crossing with a 1-resolution.  Thus $\D_i$ will have all crossings up to the $i^\text{th}$ resolved as 0-resolutions, while $\E_i$ will have all crossings up to the $(i-1)^\text{st}$ resolved as 0-resolutions, but the $i^\text{th}$ as a 1-resolution.  This arrangement allows us to see, at each step $i$, the cofibration sequence (ignoring the homological shifts)
\begin{equation}
\label{pos cofib seq}
\X^{j+N_{\D_i}+k\tau_{\D_i}}\left( \la \T_n^k , \D_i \ra \right) \hookrightarrow
\X^{j+N_{\D_{i-1}}+k\tau_{\D_{i-1}}}\left( \la \T_n^k, \D_{i-1} \ra \right) \twoheadrightarrow
\X^{j+N_{\E_i}+k\tau_{\E_i}-1}\left( \la \T_n^k , \E_i \ra \right).
\end{equation}
For further clarification, see Figure \ref{DandEposseq}.  Notice the subscripts on the $\tau$ terms - the orientations (and thus positive/negative crossing information) of the strands within $\T_n^k$ may change when resolving crossings (see Remark \ref{tau N rmk}).  However, we also know from Lemma \ref{NZlemma} that all of the $\tau_*$ terms and $N_*$ terms are \emph{independent of $k$}.  The final term $\D_{n(n-1)}$ is precisely $\Z$, so Lemma \ref{NZlemma} allows the $\tau$ and $N_{\Z}$ terms to be preserved as indicated in the sequence (\ref{pos seq}).  The suspensions giving the homological shifts are clear: we are counting the number of negative crossings introduced in a new twist.

\begin{figure}
\centering
\includegraphics[scale=.2]{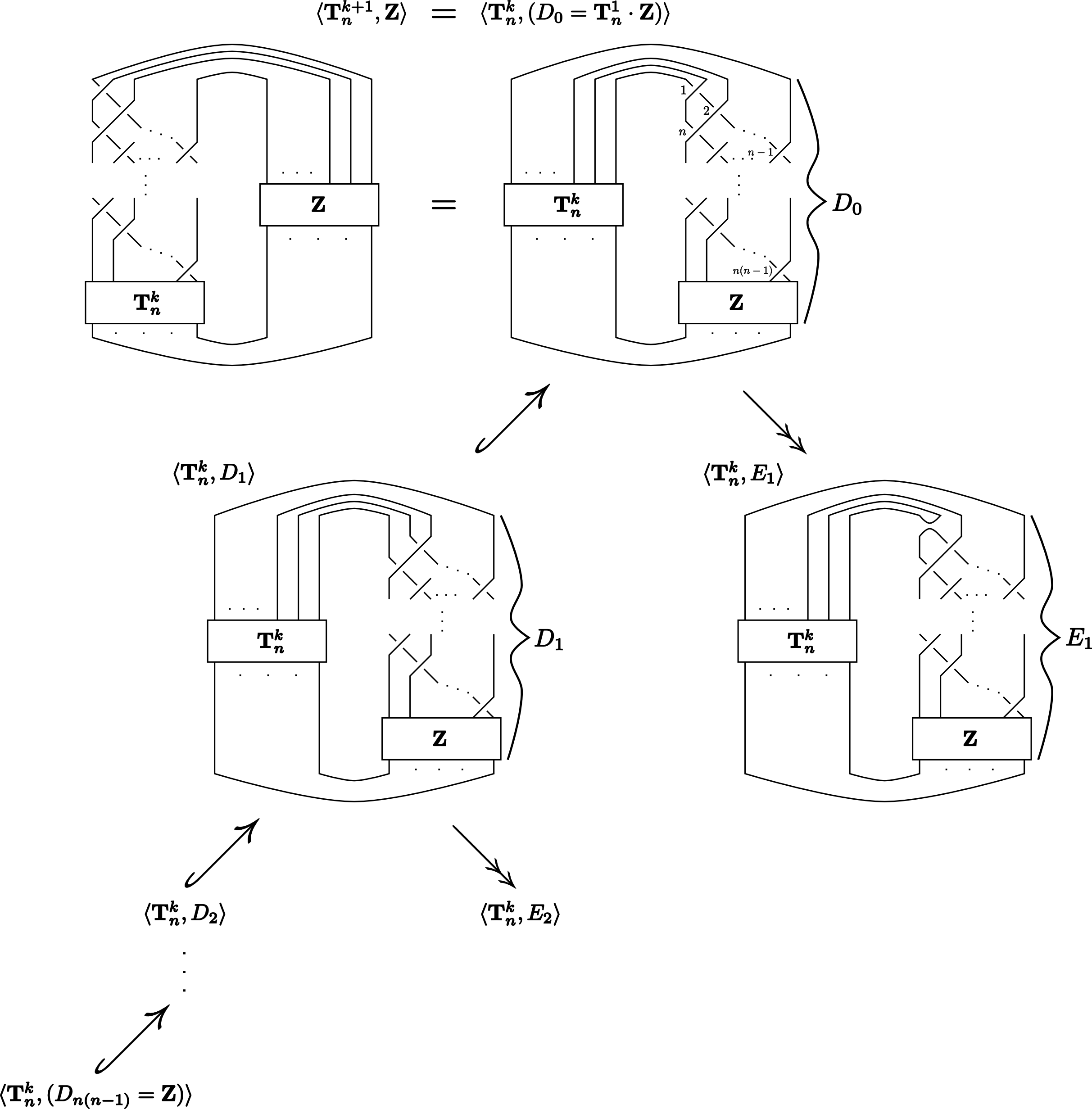}
\caption{Building a single map in the sequence (\ref{pos seq}) as a composition of inclusions coming from resolving crossings as in Lemma \ref{Lemma3.32}.  The numbering on the crossings in the diagram $\la \T_n^k , D_0 \ra$ indicates the order in which we resolve crossings.  $D_1$ and $E_1$ are illustrated as well, with the first crossing resolved.  Note that $E_2$ is obtained from $D_1$, not from $E_1$.  Thus any $E_i$ will have precisely one cup/cap.}
\label{DandEposseq}
\end{figure}

The left-handed sequence (\ref{neg seq}) is built using compositions of the surjections of the cofibration sequence (\ref{Lemma3.32onX}), since the left-handed twist $\T_n^{-1}$ needs an all-one resolution to give the identity braid $\I_n$.  For this reason, the roles of the $\D_i$ and $\E_i$ are swapped, and their definitions change slightly.  To prevent confusion, we use new names $\F_i$ and $\G_i$ and define $\G_0:=\T_n^{-1}\cdot\Z$, and let $\G_i$ denote the diagram obtained from $\G_{i-1}$ by resolving the $i^\text{th}$ crossing with a 1-resolution, and let $\F_i$ denote the diagram obtained from $\G_{i-1}$ by resolving the $i^\text{th}$ crossing with a 0-resolution.  Pictorially the $\G_i$ match the $\D_i$ from above, and the $\F_i$ match the $\E_i$, but in the cofibration sequences we see (again ignoring homological shifts)
\begin{equation}
\label{neg cofib seq}
\begin{aligned}
\X^{j+(k+1)n(n-1)-(i-1)+N_{\F_i}+k\tau_{\F_i}}\left( \la \T_n^{-k} , \F_i \ra \right) &\hookrightarrow
\X^{j+(k+1)n(n-1)-(i-1)+N_{\G_{i-1}}+k\tau_{\G_{i-1}}}\left( \la \T_n^{-k}, \G_{i-1} \ra \right)\\ &\twoheadrightarrow
\X^{j+(k+1)n(n-1)-i+N_{\G_i}+k\tau_{\G_i}}\left( \la \T_n^{-k} , \G_i \ra \right).
\end{aligned}
\end{equation}
Notice the extra shifts of $i-1$ and $i$, which occur because we have been `losing' 1-resolutions along the way.  We can see that, once $i=n(n-1)$, we arrive at $j+kn(n-1)$ together with the normalization terms, as desired for $\la \T_n^{-k},\Z\ra$ in sequence (\ref{neg seq}).  We use Lemma \ref{NZlemma} in the same way to guarantee that the $N_{\Z}$ and $\tau$ terms don't change, and we also see the extra homological shift due to losing 1-resolutions as we go.
\end{proof}

\begin{proposition}
\label{main seq stab}
Fix $j\in\mathbb{Z}$ and $n\geq 2$ in $\N$.  Then for any $(n,n)$-tangle $\Z$, both sequences (\ref{pos seq}) and (\ref{neg seq}) stabilize.  That is, there exist bounds $b^+$ and $b^-$ such that, for $k\geq b^+$, the maps in (\ref{pos seq}) are all stable homotopy equivalences, and similarly for $k\geq b^-$ for the maps in (\ref{neg seq}).  Furthermore, $b^+$ depends only on $j$ and the all-zero resolution of $\Z$, while $b^-$ depends on $j$, the number of crossings in $\Z$, and the all-one resolution of $\Z$.
\end{proposition}

\begin{proof}
We will prove the stabilization of the two sequences separately to highlight the slight differences between the two.  The notations $\D_i$, $\E_i$, $\F_i$ and $\G_i$ introduced in the previous proof will be used throughout.  Both cases will be similar to the arguments in \cite{MW}. 

Focusing first on the right-handed case, we consider the cofibration sequences (\ref{pos cofib seq}).  According to Lemma \ref{Lemma3.32}, as long as all of the Khovanov chain complexes $KC^{j+N_{\E_i}+k\tau_{\E_i}-1}\left( \la \T_n^k , \E_i \ra \right)$ are acyclic, the inclusions in Equation (\ref{pos cofib seq}) will be stable homotopy equivalences for all $i=1,\dots,n(n-1)$, allowing us to conclude that their composition (which is the map in the sequence (\ref{pos seq})) is as well.  Let $\min_q(\cdot)$ be the minimal $q$-degree of non-zero Khovanov homology for a link diagram.  Our goal now is to find a bound $b^+$ so that, for all $i=1,\dots,n(n-1)$,
\begin{equation}
\label{b+ goal}
j+N_{\E_i}+k\tau_{\E_i}-1 < \min_q\left( \la \T_n^k , \E_i \ra \right) \hspace{.3in} \text{for all }k\geq b^+.
\end{equation}

Figure \ref{EitoEi'} illustrates the key point of the proof.  The diagram $\la \T_n^k , \E_i \ra$ has a turnback at the `top' of $E_i$ that can be swung around and `pulled through' the twisting $\T_n^k$ and then back around to the bottom of $E_i$, just as in Lemma \ref{Counting lemma}.  Let $E'_i$ denote the resulting tangle, so that we have $\la \T_n^k , \E_i \ra$ isotopic to $\la \T_{n-2}^k , \E'_i \ra$.  Since Khovanov homology is an isotopy invariant, we must have
\begin{equation}
\label{Ei minq isotopy}
\min_q\left( \la \T_n^k , \E_i \ra \right) = \min_q\left( \la \T_{n-2}^k , \E'_i \ra\right).
\end{equation}

\begin{figure}
\centering
\vspace{.2in}
\includegraphics[scale=.4]{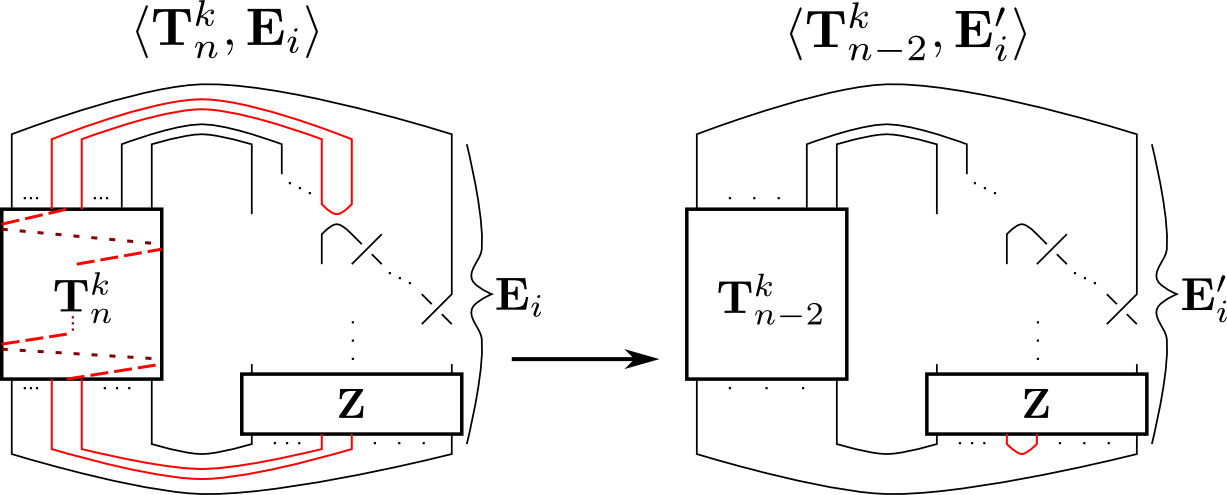}
\caption{Pulling the turnback in $\la \T_n^k , \E_i \ra$ through the twists to get $\la \T_{n-2}^k , \E'_i \ra$.  The turnback and its path are indicated in red.  Note that none of the crossings in $\mathbf{E_i}$ (including $\mathbf{Z}$) are affected.}
\vspace{.2in}
\label{EitoEi'}
\end{figure}

Now just as in the proof of Lemma 3.12 in \cite{MW}, the minimal $q$-degree of non-zero Khovanov homology for a diagram is bounded below by the minimal possible $q$-degree in the entire Khovanov chain complex, which occurs in the all-zero resolution by decorating all of the circles with $v_-$.  The all-zero resolution of the crossings coming from $\T$'s give identity braids, and so we have
\begin{equation}
\min_q\left( \la \T_{n-2}^k , \E'_i \ra \right) \geq -\#\text{circ}\left( \la \I_{n-2} , \Z^{\cap \iota}_{\cup \iota , \text{all-zero} } \ra \right) + k\tau'_{\E_i'} + N_{\E'_i}.
\end{equation}
Here $\#\text{circ}(\cdot)$ indicates the number of circles present in the planar diagram, while $\iota:= i \mod(n-1)$.  The ``all-zero'' subscript indicates that all of the crossings in $\Z^{\cap \iota}_{\cup \iota}$ have been resolved into zero-resolutions.  The $\tau'$ term and the $N$ term are the $n^+-2n^-$ normalization terms as usual.  The $\tau'$ indicates that we are counting positive and negative crossings from $\T_{n-2}$ as opposed to $\tau$ that was counting such crossings in $\T_n$ (see the notation used in Lemma \ref{Counting lemma}).

Now because we performed an isotopy to get from $\E_i$ to $\E'_i$, the orientations of the strands did not change.  Furthermore, no crossings were added to or removed from $\E_i$.  Thus $N_{\E_i}=N_{\E'_i}$, and $\tau'_{\E'_i}$ can be viewed as $\tau'_{\E_i}$.  We then use Equation (\ref{pos tau'}) from Lemma \ref{Counting lemma} to deduce
\begin{equation}
\label{pos minq bound}
\min_q\left( \la \T_{n-2}^k , \E'_i \ra \right) \geq -\#\text{circ}\left( \la \I_{n-2} , \Z^{\cap \iota}_{\cup \iota , \text{all-zero} } \ra \right) + k(\tau_{\E_i}+2n) + N_{\E_i}.
\end{equation}
Combining Equations (\ref{b+ goal}), (\ref{Ei minq isotopy}), and (\ref{pos minq bound}) gives us the following new goal for our bound $b^+$:
\begin{equation}
\label{b+ new goal}
j+N_{\E_i}+k\tau_{\E_i}-1 <  -\#\text{circ}\left( \la \I_{n-2} , \Z^{\cap \iota}_{\cup \iota , \text{all-zero} } \ra \right) + k(\tau_{\E_i}+2n) + N_{\E_i} \hspace {.2 in} \text{for all }k\geq b^+
\end{equation}
which is clearly satisfied for all $i$ by choosing
\begin{equation}
\label{b+ definition}
b^+:= \max_{\iota=1,\dots,n-1}\frac{j + \#\text{circ}\left( \la \I_{n-2} , \Z^{\cap \iota}_{\cup \iota , \text{all-zero} } \ra \right)}{2n}.
\end{equation}
We see clearly from the definition of $b^+$ that it depends only on $j$ and the all-zero resolution of $\Z$, as claimed.  The final homological shift is clear.

We now turn to the left-handed sequence (\ref{neg seq}).  The strategy is very similar so we will be brief.  This time we consider the cofibration sequence (\ref{neg cofib seq}) where our goal is to bound $k$ to ensure that all of the $KC^{j-(i-1)+N_{\F_i}+k\tau_{\F_i}}\left( \la \T_n^{-k} , \F_i \ra \right)$ are acyclic, ensuring that the surjections give stable homotopy equivalences.

Since the $\F_i$ pictorially match the $\E_i$ from before, we can still use Lemma \ref{Counting lemma} in the same way to arrive at $\la \T_{n-2}^{-k}, \F'_i \ra$ with corresponding $\tau'$.  Now comes the main difference between the left- and right-handed sequences.  For the right-handed twist, the all-zero resolution of $\T_{n-2}^k$ is just $\I_n$; in particular, it is independent of $k$.  Taking 0-resolutions motivates bounding based on the minimal $q$-degree.  But for the left-handed twist, it is the all-one resolution of $\T_n^{-k}$ that is just $\I_n$.  Taking 1-resolutions motivates bounding based on the maximal $q$-degree.  So we define $\max_q(\cdot)$ to be the maximal $q$-degree of non-zero Khovanov homology for a given diagram, which is bounded above by the maximal $q$-degree for the full Khovanov chain complex.  Following the logic of the right-handed case, we get
\begin{align*}
\max_q\left( \la \T_n^{-k} , \F_i \ra \right) &= \max_q\left( \la \T_{n-2}^{-k}, \F'_i \ra \right)\\
&\leq \#\text{cros}\left( \la \T_{n-2}^{-k}, \F'_i \ra \right) + \#\text{circ}\left( \la \I_{n-2} , \Z^{\cap \iota}_{\cup \iota , \text{all-one} } \ra \right) + k(\tau'_{\F'_i}) + N_{\F_i}\\
&= \left( k(n-2)(n-3) + (n(n-1)-i) + \#\text{cros}(\Z) \right) +  \#\text{circ}\left( \la \I_{n-2} , \Z^{\cap \iota}_{\cup \iota , \text{all-one} } \ra \right)\\
&\hspace{.2in}+ k(\tau_{\F_i}+2n-6) + N_{\F_i}\\
&= k(n^2-3n+\tau_{\F_i}) + (n(n-1)-i) + \#\text{cros}(\Z) + \#\text{circ}\left( \la \I_{n-2} , \Z^{\cap \iota}_{\cup \iota , \text{all-one} } \ra \right)\\
&\hspace{.2in}+ N_{\F_i}.
\end{align*}
The $\#\text{cros}(\cdot)$ denotes the total number of crossings.  This term appears because the $q$-degree counts the number of 1-resolutions taken (which will be all of the crossings).  The third line breaks this term into several self-explanatory pieces; the $n(n-1)-i$ term handles the crossings `above' the $\Z$ (See Figure \ref{EitoEi'}).  Meanwhile, the changing of $\tau'$ to $\tau+2n-6$ in the third line follows from the left-handed version of Lemma \ref{Counting lemma}.

From the cofibration (\ref{neg cofib seq}) above, we see that our goal for the left-handed twists is to ensure that, for all $i=1,\dots,n(n-1)$, and for all $k\geq b^-$,
\begin{align*}
j+(k+1)&n(n-1)-(i-1)+N_{\F_i}+k\tau_{\F_i} >\\
&k(n^2-3n+\tau_{\F_i}) + (n(n-1)-i) + \#\text{cros}(\Z) + \#\text{circ}\left( \la \I_{n-2} , \Z^{\cap \iota}_{\cup \iota , \text{all-one} } \ra \right) + N_{\F_i}.
\end{align*}
This is clearly achieved by setting
\begin{equation}
\label{b- definition}
b^-:=\max_{\iota=1,\dots,n-1}\frac{-j+ \#\text{cros}(\Z) + \#\text{circ}\left( \la \I_{n-2} , \Z^{\cap \iota}_{\cup \iota , \text{all-one} } \ra \right)}{2n}
\end{equation}
which clearly depends only on $j$, the number of crossings in $\Z$, and the all-one resolution of $\Z$ as desired.
\end{proof}

\begin{remark}
\label{b+ vs b-}
Notice the similarity between the bounds $b^+$ and $b^-$.  In both cases, the bound involves $\frac{\pm j}{2n}$ plus a constant term (independent of $j$).  Thus all of the careful tracking of normalization shifts `cancel out' in precisely the same way regardless of using right- or left-handed twists.  The sign change of $j$ versus $-j$ also makes sense when we recall that the graded Euler characteristic of these spaces is meant to give a power series expansion of the corresponding rational functions coming from the Jones-Wenzl projectors, in $q$ for right-handed twists (so using positive $j$ terms) and in $q^{-1}$ for left-handed twists (so using negative $j$ terms).  With this in mind, the only real difference between $b^+$ and $b^-$ comes from the use of the all-zero resolution of $D$ versus the all-one resolution, and the need to count crossings away from the left-handed twists.
\end{remark}

We are now ready to prove Theorem \ref{Projectors X exists}.  Let $D$ denote a diagram obtained from a link diagram by formally replacing a finite number of identity braids $\I_{n_i}$ with Jones-Wenzl projectors $P_{n_i}$.  Let $m\in\N$ denote the total number of projectors in $D$.  For any $(k_1,\dots,k_m)\in(\N\cup 0)^m$, let $D^{\pm}(k_1,\dots,k_m)$ denote the diagram $D$ with each $P_{n_i}$ replaced by $T_{n_i}^{\pm k_i}$.  Note that it is very important that the diagrams have either all right-handed twists, or all left-handed twists.  We do not allow any mixing of the two.

We focus on the right-handed case first.  Fixing $j\in\mathbb{Z}$, we consider the infinite $m$-dimensional cube of maps built as follows.  The vertices of the cube correspond to $(k_1,\dots,k_m)\in(\N\cup 0)^m$.  At each such vertex we place the space
\begin{equation}
\label{pos cube vx}
\X_+^{j+N_D}(k_1,\dots,k_m):=\Sigma^{-\sum_{i=1}^m k_i\eta_i}\X^{j+N_D+\sum_{i=1}^m k_i\tau_i}(D(k_1,\dots,k_m)).
\end{equation}
Here, the subscripts on the normalization shifts $\tau$ and $\eta$ indicate which $\T_{n_i}$ is being referred to (see Definition (\ref{tau N def}) and Remark (\ref{tau N rmk})).  Meanwhile, the $N_D$ is referring to the normalization shift $n^+-2n^-$ for all crossings totally separate from any of the inserted twists (ie, crossings present in the original diagram $D$, discounting the Jones-Wenzl projectors).  Now between any two adjacent vertices of $(\N\cup 0)^m$, we see all of the $k_i$ remain constant except one of them, say $k_{\hat{i}}$, which differs by one between the two vertices.  To this edge we assign the map
\begin{equation}
\label{pos cube edge}
\X_+^{j+N_D}(k_1,\dots,k_{\hat{i}},\dots,k_m) \hookrightarrow \X_+^{j+N_D}(k_1,\dots,k_{\hat{i}}+1,\dots,k_m)
\end{equation}
induced by Lemma \ref{Lemma3.32} as in Proposition \ref{building seq}.  

\begin{definition}
\label{X+ def}
Given a diagram $D$ involving Jones-Wenzl projectors, the \emph{(right-handed) Khovanov homotopy type of $D$} is defined to be the wedge sum $\X_+(D):=\bigvee_{j\in\Z}\X_+^{j+N_D}(D)$ where for each $q$-degree $j+N_D$, the spectrum $\X_+^{j+N_D}(D)$ is defined to be the homotopy colimit of the cube of maps described by Equations (\ref{pos cube vx}) and (\ref{pos cube edge}).
\end{definition}

\begin{proof}[Proof of Theorem \ref{Projectors X exists} (Right-handed Case)]
We wish to show that the cube of maps defining $\X_+^{j+N_D}(D)$ `stabilizes' in a particular sense.  To do this we isolate a single projector $P_{n_{\hat{i}}}$ and fix all of the $k_{i\neq \hat{i}}$.  This allows us to view the maps (\ref{pos cube edge}) as (ignoring homological shifts)
\begin{equation}
\label{isolate pos Pn}
\X^{j+k_{\hat{i}}\tau_{\hat{i}} + N_D +\sum_{i\neq\hat{i}}k_i\tau_i} \left( \la \T_{n_{\hat{i}}}^{k_{\hat{i}}} , \Z \ra \right)
\hookrightarrow
\X^{j+(k_{\hat{i}}+1)\tau_{\hat{i}} + N_D +\sum_{i\neq\hat{i}}k_i\tau_i} \left( \la \T_{n_{\hat{i}}}^{k_{\hat{i}}+1} , \Z \ra \right)
\end{equation}
where the tangle $\Z$ includes all of the other $\T_{n_i}^{k_i}$.  Having fixed $j$, these maps are all stable homotopy equivalences for $k_{\hat{i}}>b^+_{\hat{i}}$ for some bound $b^+_{\hat{i}}$ that depends only on the all-zero resolution of $\Z$.  Since the all-zero resolution of any $\T_{n_i}^{k_i}$ is just $\I_{n_i}$ regardless of $k_i$, this bound $b^+_{\hat{i}}$ is \emph{independent of the other $k_i$} (this is the point that requires we do not mix right- and left-handed twists in our construction).  Thus we can find the various bounds $b^+_i$ one projector at a time effectively ignoring the rest.  Since there are only finitely many projectors, we can find a global bound $b^+$ which works for all of the $k_i$ at once and declare that the cube is stable for all $k_i>b^+$.  This also allows us to use a simpler notation: let $D(k):=D(k,\dots,k)$, and similarly for $\X_+^{j+N_D}(k)=\X_+^{j+N_D}(k,\dots,k)$.  Our proof then shows that, for any fixed $j\in\Z$, the `diagonal sequence' $\X_+^{j+N(D)}(k)$ stabilizes as $k\rightarrow\infty$, and so the hocolim $\X_+^{j+N_D}(D) \simeq \X_+^{j+N_D}(k)$ for some large enough $k$ depending on $j$.  Since the chain complexes of the twists are known to stabilize to the categorified Jones-Wenzl projectors, the wedge sum $\X_+(D)=\bigvee_{j\in\Z}\X_+^{j+N_D}(D)$ satisfies the requirements of Theorem \ref{Projectors X exists}.
\end{proof}

The left-handed twists work in exactly the same fashion, so we only mention the slight differences.  We populate the vertices of the cube by spaces
\begin{equation}
\label{neg cube vx}
\X_-^{j+N_D}(k_1,\dots,k_m):=\Sigma^{-\sum_{i=1}^m k_i\eta_i}\X^{j+\sum_{i=1}^m k_in_i(n_i-1) + N_D+\sum_{i=1}^m k_i\tau_i}(D(k_1,\dots,k_m))
\end{equation}
and the edges are maps
\begin{equation}
\label{neg cube edge}
\X_-^{j+N(D)}(k_1,\dots,k_{\hat{i}},\dots,k_m) \twoheadleftarrow \X_-^{j+N_D}(k_1,\dots,k_{\hat{i}}+1,\dots,k_m)
\end{equation}
induced by Lemma \ref{Lemma3.32} once again.  Notice the extra grading shift $\sum_{i=1}^m k_in_i(n_i-1)$, which counts the number of crossings available in all of the $\T_{n_i}^{-k_i}$.

\begin{definition}
\label{X- def}
Given a diagram $D$ involving Jones-Wenzl projectors, the \emph{(left-handed) Khovanov homotopy type of $D$} is defined to be the wedge sum $\X_-(D):=\bigvee_{j\in\Z}\X_-^{j+N_D}(D)$ where for each $q$-degree $j+N_D$, the spectrum $\X_-^{j+N_D}(D)$ is defined to be the homotopy colimit of the cube of maps described by Equations (\ref{neg cube vx}) and (\ref{neg cube edge}).
\end{definition}

\begin{proof}[Proof of Theorem \ref{Projectors X exists} (Left-handed Case)]
Focusing on one projector ($\hat{i}$) at a time as before, the formula (\ref{b- definition}) for $b^-_{\hat{i}}$ does appear to depend on the other $k_i$ since the term $\#\text{cros}(\Z)$ will count crossings in the other twists.  However, this count is cancelled out precisely by the extra grading shift $\sum_{i=1}^m k_in_i(n_i-1)$, and the bounds $b_i^-$ are again mutually independent allowing the same argument as for the right-handed case to go through.  The details here are left to the reader.
\end{proof}

Thus we have two equally eligible candidates, $\X_+(D)$ and $\X_-(D)$, for a spectrum that satisfies the requirements of Theorem \ref{Projectors X exists}, depending on whether we want to view the Euler characteristic as a power series representation of the corresponding rational function in $q^{+1}$ or $q^{-1}$.  In either case, the wedge summand in a specific $q$-degree can be computed using a finite-twist approximation $D(k)$ where the amount of twisting $k$ needed depends both on the diagram $D$ and on the $q$-degree being considered.

\begin{remark}
\label{proj one at a time}
The independence of the various $k_i$ used in the proofs above has been used to take the homotopy colimit `diagonally', simplifying the notation by tracking only a single value of $k$.  However, this independence can also be viewed as allowing us to take the colimit one projector at a time, in any order we like.  This is already implicit in the diagonal version in the passage from $D(k)$ to $D(k+1)$, where it does not matter in what order we treat all of the projectors going from their individual $k$-twists to their individual $(k+1)$-twists.
\end{remark}

\subsection{Properties of $\X(D)$}
Before going on to establish the connection to spin networks and colored links, we state and prove some properties for $\X_+(D)$ for diagrams $D$ with Jones-Wenzl projectors as above.  The propositions in this section will be stated and proved for right-handed twists only; the left-handed versions for $\X_-(D)$ are proved analogously, using alterations similar to those discussed in the previous section.  As such, we drop the $+$ notation for the time being.  In addition, as seen in the proofs in the previous section, the homological shifts are irrelevant from the point of view of establishing results about stabilizations of sequences, and so we will also be ignoring the various suspensions throughout the proofs in this section.

All of the proofs in this section follow the same pattern.  We prove properties of $\X(D)$ one $q$-degree at a time.  To do so, we replace $D$ by the corresponding finite-twist approximation $D(k>b^+)$ as in the proof of Theorem \ref{Projectors X exists}, and focus on one particular projector, say $P_{n_1}$ (now replaced by $\T_{n_1}^k$).  We prove the property we are interested in using this `honest' link diagram (no projectors) $D(k)$, and then conclude the same result about $\X(D)$.

Our first property is perhaps the most fundamental one.  Recall that the first axiom used to characterize both the Jones-Wenzl projectors and their categorifications is that they are `killed by turnbacks'.  The following proposition gives the analogous statement for our spectra $\X(D)$.

\begin{proposition}
\label{X killed by turnbacks}
For any diagram $D$ involving at least one Jones-Wenzl projector that is capped by at least one turnback, $\X(D)\simeq *$.
\end{proposition}
\begin{proof}
Fix $j\in\mathbb{Z}$ and focus in on a projector that is capped by a turnback.  Label this projector $P_{n_1}$.  As in the proof of Theorem \ref{Projectors X exists}, let $m$ denote the total number of projectors, and replace the diagram $D$ by the diagram $D(k)$ with $\T_{n_i}^{k}$ in place of all of the projectors.  Singling out the $\T_{n_1}^k$ allows us to consider the diagram $D(k)=\la (\T_{n_1}^k)^{\cap \iota} , \Z \ra$ for some $\iota\in\{1,\dots,n_1-1\}$, where the $\Z$ contains all of the other $\T_{n_i}^k$ (the picture is precisely that of Figure \ref{CountingLemmaPic}).  The corresponding term in the diagonal sequence is $\X^{j + N_D + k\tau_1 + k\sum_{i=2}^m \tau_i}(D(k)) = \X^{j+k\tau_1+N_\Z}\left( \la (\T_{n_1}^k)^{\cap \iota} , \Z \ra \right)$, where $N_\Z$ is accounting for both $N_D$ (crossings away from any twists) and $k\sum_{i=2}^m \tau_i$ (crossings from the other twists).  We use Lemma \ref{Counting lemma} and pull the cap through the twisting and around to the bottom of $\Z$ so that
\[\X^{j+k\tau_1+N_\Z}\left( \la (\T_{n_1}^k)^{\cap \iota} , \Z \ra \right) \simeq \X^{j+k\tau_1+N_\Z}\left( \la (\T_{n_1-2}^k) , \Z_{\cup n_1-\iota} \ra \right).\]
Here note that the $\tau_1$ refers to counting crossings in $\T_{n_1}$ on both sides of this equivalence, despite the fact that $\T_{n_1}$ has been replaced with $\T_{n_1-2}$ on the right side.  Meanwhile, as in the proof of Proposition \ref{main seq stab}, we can calculate
\[\min_q\left( \la (\T_{n_1-2}^k) , \Z_{\cup n_1-\iota} \ra \right) \geq
-\#\text{circ}\left( \la \I_{n_1-2} , \Z_{\cup n_1-\iota , \text{all-zero} } \ra \right) + k\tau'_1 + N_{\Z}\]
where $\tau'_1$ now refers to counting crossings in $\T_{n_1-2}$ as in Lemma \ref{Counting lemma}, which tells us that $\tau'_1=\tau_1+2n_1$.  Thus for $k > \frac{j+\#\text{circ}\left( \la \I_{n_1-2} , \Z_{\cup n_1-\iota , \text{all-zero} } \ra \right)}{2n_1}$ we can perform some simple algebra to conclude
\[j+k\tau_1+N_\Z < \min_q\left( \la (\T_{n_1-2}^k) , \Z_{\cup n_1-\iota} \ra \right)\]
forcing
\[\X^{j + N_D + k\tau_1 + k\sum_{i=2}^m \tau_i}(D(k))=
\X^{j+k\tau_1+N_\Z}\left( \la (\T_{n_1}^k)^{\cap \iota} , \Z \ra \right) \simeq *\]
and so the sequence needed to build $\X^{j+N_D}(D)$ stabilizes to a trivial homotopy type as $k\rightarrow\infty$, regardless of $j$.  We can conclude that 
\[\X(D) = \bigvee_j \X^{j+N_D}(D) \simeq \bigvee_j (*) \simeq *\]
as desired.  Note that the independence of the various twisting, allowing a single $k$ to be used, is implicit in this proof and allows our estimates to essentially ignore the other projectors in $D$.
\end{proof}

\begin{corollary}
\label{X straighten braids}
For any diagram $D$ involving an $n$-strand Jones-Wenzl projector $P_n$ concatenated with a braid $\beta$ on those $n$ strands, $\X^{j+N_D}(D) \simeq \Sigma^a\X^{j+N_{D\setminus\beta}-\beta^-}(D\setminus\beta)$ where $D\setminus\beta$ is used to denote the diagram created by replacing $\beta$ with $\I_n$, the identity braid on those same $n$ strands (this replacement is referred to as \emph{straightening the braid $\beta$}), and $\beta^-$ is the number of crossings of the form \ncrossing in $\beta$ viewed vertically (ie the number of crossings that require 1-resolutions to transform $\beta$ into $\I_n$).  The homological shift $a$ is the difference between the number of negative crossings in the two diagrams, as in Lemma \ref{Lemma3.32}.
\end{corollary}
\begin{proof}
Fix $j\in\mathbb{Z}$.  Let $P_{n_1}$ be the projector with $\beta$ concatenated.  Since any braid $\beta$ is a product of elementary generators $\sigma_\iota^{\pm 1}$ in the braid group $B_{n_1}$ (so $\iota\in\{1,\dots,n_1-1\}$), it is enough to prove the statement for such generators (ie, for a single crossing above the $P_{n_1}$).

Again, we consider $k$ large enough so that
\[\X^{j+N_D}(D) \simeq \X^{j+N_D+k\sum_{i=1}^m \tau_i} (D(k)).\]
We focus in on $\T_{n_1}^k$ and see
\[\X^{j+N_D+k\sum_{i=1}^m \tau_i} (D(k)) \simeq
\X^{j+N_D+k\sum_{i=1}^m \tau_i}\left( \la \sigma_\iota^{\pm 1} \cdot \T_{n_1}^k , \Z \ra \right).\]
Now we use Lemma \ref{Lemma3.32} to resolve the crossing $\sigma_\iota^{\pm 1}$ creating a cofibration sequence.  From the proof of Proposition \ref{X killed by turnbacks}, we see that for large enough $k$ (and once $k>b^+$, we are always free to increase it without changing the stable homotopy type), whichever resolution has a turnback contributes a trivial term to the cofibration sequence, so that Lemma \ref{Lemma3.32} ensures that the homotopy type of the `straight' resolution is stably homotopy equivalent to the original, up to precisely the grading shifts indicated in the statement of this Corollary.  The extra $-1$ for $\sigma_\iota^{-1}$ (which builds the $-\beta^-$ term) is also clear since in this case, it is the 1-resolution that is the `straight' resolution.  The proof works for all $j$, and so the result follows.
\end{proof}

The next corollary can be viewed as lifting the idempotency of the Jones-Wenzl projectors and their categorifications to the realm of the stable homotopy types.
\begin{corollary}
\label{X idempotent}
Let $D$ be a diagram involving two concatenated projectors of possibly different sizes, say $P_{n_1}\cdot P_{n_2}$ with $n_1\leq n_2$ (see Figure \ref{Proj Concat} for clarification on this notion).  Let $D'$ be obtained from $D$ by replacing the smaller projector $P_{n_1}$ with an identity braid $\I_{n_1}$.  Then $\X(D)\simeq\X(D')$.
\end{corollary}
\begin{figure}
\centering
\vspace{.1in}\includegraphics[scale=.4]{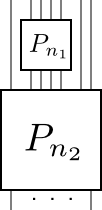}
\caption{An example of two concatenated projectors $P_{n_1}\cdot P_{n_2}$ with $n_1\leq n_2$, for which Corollary \ref{X idempotent} allows us to absorb $P_{n_1}$ into $P_{n_2}$ on the level of the homotopy types.}
\vspace{.1in}
\label{Proj Concat}
\end{figure}
\begin{proof}
As usual, we fix $j\in\mathbb{Z}$ and replace $D$ by $D(k)$ for $k>b^+$ as in the proof of Theorem \ref{Projectors X exists}.  This time however we will make stronger use of the independence of the various $k_i$ to fix $k_1>b^+_1$, while still allowing the other $k_i$ to limit towards infinity together.  In symbols, we are considering $\X^{j+N(D)+k_1\tau_1 +k\sum_{i=2}^m \tau_i}(D(k_1,k,\dots,k))$.  Having fixed $k_1$ in this way, we can view the $\T_{n_1}^{k_1}$ as a braid that is allowed to be straightened as in Corollary \ref{X straighten braids}.  When doing this, the grading shift effectively removes the $k_1\tau_1$, and there are no $-1$ terms because all of the crossings are of the form \crossing.  This leaves us with precisely 
\[\X^{j+N(D)+k_1\tau_1 +k\sum_{i=2}^m \tau_i}(D(k_1,k,\dots,k))\simeq
\X^{j+N(D)+k\sum_{i=2}^m \tau_i}(D'(k))\]
and since the $\T_{n_1}^{k_1}$ contributed only full twists to the diagram, the strand orientations before and after the straightening can be the same so that $N(D)=N(D')$.  Thus we are left with $\X^{j+N(D')+k\sum_{i=2}^m \tau_i}(D'(k))$ which is precisely the sequence needed to build $\X(D')$.
\end{proof}

\section{Applications to Quantum Spin Networks and Colored Links}
The aim of this section is to provide the necessary background in order to view Theorems \ref{Colored X exists} and \ref{Spin Network X exists} as corollaries of Theorem \ref{Projectors X exists}, and then to prove them accordingly.  In short, the `proof' for both statements is that categorified quantum invariants of spin networks and colored links are defined using diagrams involving Jones-Wenzl projectors, for which Theorem \ref{Projectors X exists} supplies a well-defined Khovanov homotopy type.  In the case of colored links, the proof of invariance requires only a few more remarks related to Reidemeister moves and framing.  The reader who is already familiar with these subjects can safely skim this section, although the notation used for colored links will be used again in the following sections related to tails.  We also state and prove a property about the colored Khovanov homotopy type of a 1-colored unknot linking as simply as possible with another colored link related to some discussions in \cite{Hog}.

\subsection{Quantum Spin Networks}

A (closed) quantum spin network (the notion dates back to Roger Penrose in \cite{Pen}) consists of a trivalent graph where each edge has been labelled with a natural number.  The labels are not entirely independent: for each vertex where three edges labelled $n_1,n_2,n_3$ meet, we must have
\begin{equation}
\label{spin network conditions}
\begin{gathered}
n_i \leq n_j+n_k \text{ }\forall\text{ } \{i,j,k\}=\{1,2,3\}\\
n_1+n_2+n_3 \equiv 0 \mod 2.
\end{gathered}
\end{equation}

From such a spin network $G$, a $q$-deformed quantum invariant can be defined as follows (see chapter 4 in \cite{KL}).  First we replace each $n$-labelled edge by a cable of $n$ parallel strands together with a copy of the Jones-Wenzl projector $P_n$.  Then we replace each vertex having edge labels $n_1,n_2,n_3$ with a `balanced splitting' of the cables as in Figure \ref{Spin vertex}.  Call the resulting diagram $D(G)$.  The final step is to evaluate the Jones polynomial of $D(G)$, using the rational expressions for the Jones-Wenzl projectors present.

\begin{figure}[h]
\centering
\vspace{.1in}\includegraphics[scale=.4]{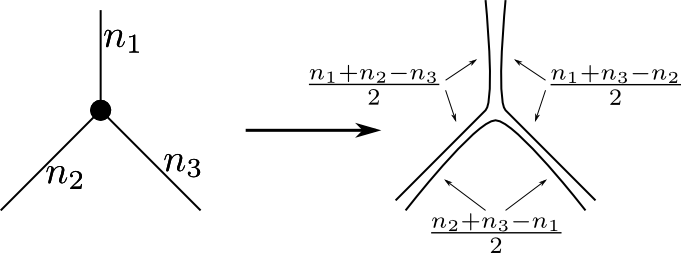}
\caption{Building the $q$-deformed invariant of a quantum spin network; the $n_i$ are labels in the original network, and the fractions on the right hand side tell how many parallel strands to send each direction from the vertex}
\vspace{.1in}
\label{Spin vertex}
\end{figure}

In \cite{CK}, Cooper and Krushkal replace the projectors in $D(G)$ with their own categorified projectors, thus defining a categorified spin network.  If instead of this we replace the projectors with Rozansky's categorifications using infinite twists, we see a diagram of the form covered by Theorem \ref{Projectors X exists}.

\begin{definition}
\label{Spin Network X def}
Given a quantum spin network $G$, we define the \emph{Khovanov homotopy type of the spin network $G$} to be $\X(G):=\X(D(G))$ as defined in Theorem \ref{Projectors X exists} for the diagram $D(G)$.
\end{definition}

\begin{proof}[Proof of Theorem \ref{Spin Network X exists}]
$\X(G)$ as defined using Theorem \ref{Projectors X exists} is clearly well-defined with regards to isotopies of the graph of $G$, which induce isotopies of $D(G)$.  There is also a `twist' move at a vertex, shown in Figure \ref{Spin twist}.  This move is accomplished by a framing twist on the strand labelled $n_1$, which would result in a shift of $q$-degree for the stable homotopy type (a framing twist creates a torus braid on the relevant cable, which can be straightened at the cost of such a shift using Corollary \ref{X straighten braids}).  This corresponds to the shift described in section 4.2 of \cite{KL}.
\end{proof}

\begin{figure}[h]
\centering
\vspace{.1in}\includegraphics[scale=.4]{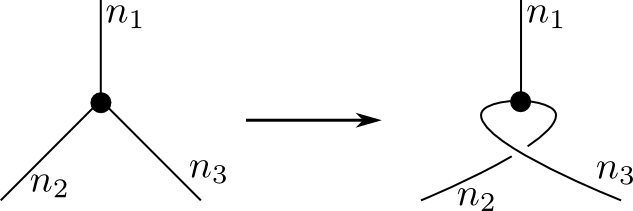}
\caption{A twist move on a spin network coming from a framing twist on the strand labelled $n_1$}
\vspace{.1in}
\label{Spin twist}
\end{figure}

\subsection{Colored Links}
A coloring of a link $L$ refers to assigning an irreducible $\mathfrak{sl}_2(\C)$ representation to each component of $L$.  Such representations are characterized by their dimension $n\in\N$, allowing us to simply consider colorings as assignments of non-negative integers to each link component.

We now describe some quantum invariants of such a link.  Let $L$ be an oriented link diagram with $\ell$ components.  For each $h=1,\dots,\ell$, we color the $h^\text{th}$ component with a natural number $n_h$.  Call such a coloring $\gamma$, and the colored link $L_\gamma$.  The \emph{colored Jones polynomial} of the link $L_\gamma$ is obtained by cabling each component with its designated $n_h$ number of strands, inserting a copy of the $n_h^\text{th}$ Jones-Wenzl projector $P_{n_h}$ into each such cabled component, and then taking the usual Jones polynomial.  See Figure \ref{Colored diagrams}.

\begin{figure}[h]
\centering
\vspace{.1in}\includegraphics[scale=.5]{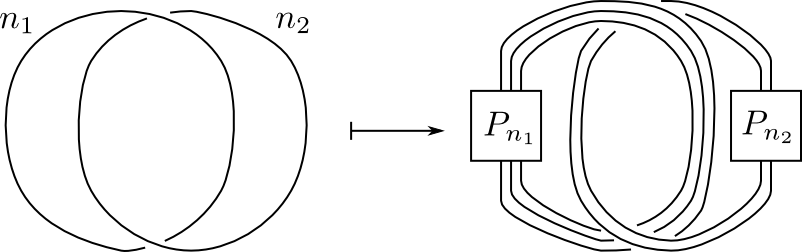}
\caption{On the left is an example $L_\gamma$, with components colored $n_1$ and $n_2$; on the right is the resulting diagram $D_{L_\gamma}$ for taking the colored Jones polynomial}
\vspace{.1in}
\label{Colored diagrams}
\end{figure}

\begin{remark}
\label{framing remark}
Note that a framing must also be designated to each component of the link to complete the definition of the invariant, but these framings can be accounted for by adding positive or negative kinks (Reidemeister 1 moves) into the diagram.  Hereafter the choice of framing will be considered specified by the presence of such kinks in the diagram for $L$, allowing all relevant information to be based on the diagram.
\end{remark}

Using categorified Jones-Wenzl projectors as in \cite{Roz} or \cite{CK} we can build a colored chain complex for the $L_\gamma$ in the same way, whose homology groups are referred to as the colored Khovanov homology of the colored link $L_\gamma$ (see \cite{CK} and \cite{Roz2}).  Using Rozansky's version of the categorified projectors allows us to prove Theorem \ref{Colored X exists}.

\begin{definition}
\label{Colored X def}
Given a colored link $L_\gamma$, we define the \emph{colored Khovanov homotopy type of $L_\gamma$} to be $\X_{c}(L_\gamma):=\X(D_{L_\gamma})$ as defined by Theorem \ref{Projectors X exists} for the diagram $D_{L_\gamma}$.
\end{definition}

\begin{proof}[Proof of Theorem \ref{Colored X exists}]
As indicated above, the colored Khovanov homology groups for a colored link $L_\gamma$ are defined by a link diagram $D_{L_\gamma}$ involving Jones-Wenzl projectors.  Therefore Theorem \ref{Projectors X exists} gives the existence of a colored Khovanov homotopy type that properly recovers the colored homology.  There is a choice of where to place the projector on each cabled component when creating $D_{L_\gamma}$.  The invariance of $\X_{c}$ with respect to such a choice is proved one $q$-degree at a time.  Since each $\X^j_c(L_\gamma)$ is equivalent to $\X^j(D_{L_\gamma}(k))$ for some large enough $k$, and $D_{L_\gamma}(k)$ is just an honest link diagram with $\T_{n_h}^k$ in place of the $P_{n_h}$, we see that these twists $\T_{n_h}^k$ can be slid up and down along the cablings, including above or below other cablings, as desired.  Similarly invariance under Reidemeister moves II and III is proved by considering the finite approximation for each $j$, where such moves give clear isotopies of honest link diagrams.  Meanwhile, Reidemeister I moves give framing shifts as expected, since undoing a kink corresponds to adding a full twist on a cable.
\end{proof}

We end this short section with a quick property of the colored Khovanov homotopy types inspired by the discussion in section 3.8 of \cite{Hog}.

\begin{definition}
\label{Hogancamp def}
Let $L_\gamma$ denote a colored link with $\ell$ components, and let $\alpha_h$ denote the component of $L_\gamma$ colored with $n_h$.  Define $L_\gamma^{o(h)}$ to be the colored link obtained from $L_\gamma$ by introducing a new unknotted, 1-colored component $\alpha_{\ell+1}$ that links positively once around the component $\alpha_h$ as in Figure \ref{Hopflike linking}.
\end{definition}

\begin{figure}[h]
\centering
\vspace{.1in}\includegraphics[scale=.4]{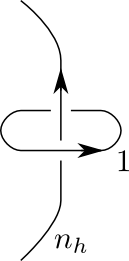}
\caption{The new unknotted, 1-colored component $\alpha_{\ell+1}$ linking positively once around the component $\alpha_h$ (colored with $n_h$), forming $L_\gamma^{o(h)}$}
\vspace{.1in}
\label{Hopflike linking}
\end{figure}

\begin{proposition}
\label{Hogancamp prop}
For any colored link $L_\gamma$ with $\ell$ components as above, the colored homotopy types of $L_\gamma$ and $L_\gamma^{o(h)}$ for any $h\in\{1,\dots,\ell\}$ fit in the following cofibration sequence:
\begin{equation}
\label{Hogancamp cofib}
\X_c^{j+1-2n_h} (L_\gamma) \hookrightarrow
\X_c^j (L_\gamma^{o(h)}) \twoheadrightarrow
\Sigma^{-2n_h}\X_c^{j-1-4n_h} (L_\gamma).
\end{equation}
\end{proposition}
\begin{proof}
We focus on $\X_c^j(L_\gamma^{o(h)})$, which we build by first cabling the components and adding in right-handed (the left-handed proof is exactly the same; see Remark \ref{hog cofib rmk} at the end of the proof) $\T_n^k$'s for large $k$, resulting in the diagram $D_{L_\gamma^{o(h)}}(k)$.  In this diagram we slide the specified $\T_{n_h}^k$ along the cabling to be drawn directly below the `new' unknot $\alpha_{\ell+1}$, which is colored by $1$ so that we need no cabling or twisting for this component.  We then construct the cofibration sequence of Lemma \ref{Lemma3.32} by resolving the `upper-left' crossing (see Figure \ref{Hopflike resolutions}).

\begin{figure}[h]
\centering
\vspace{.1in}\includegraphics[scale=.4]{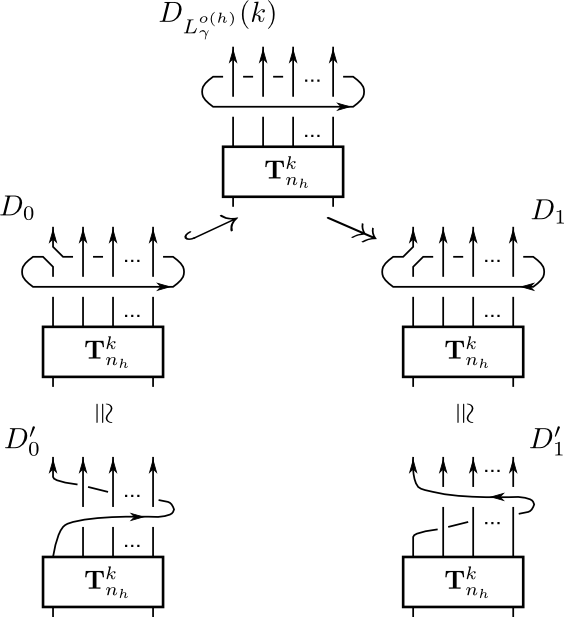}
\caption{Resolving the upper-left crossing in $D_{L_\gamma^{o(h)}}(k)$ to create a cofibration sequence.  The resulting diagrams $D_0$ and $D_1$ are isotopic to $D'_0$ and $D'_1$ which allow the use of Corollary \ref{X straighten braids}.}
\vspace{.1in}
\label{Hopflike resolutions}
\end{figure}

As illustrated in Figure \ref{Hopflike resolutions}, we denote the resulting diagrams $D_0$ and $D_1$ for the 0-resolution and 1-resolution respectively.  The 0-resolution is also the oriented one, and so the resulting shift in $q$-degree is only $-1$ for the loss of a positive crossing.  The 1-resolution allows for an orientation as shown in the diagram, where all of the previously positive crossings (there were originally $2n_h$ of them, but one was resolved) become negative.  Thus we have a $q$-degree shift of $-1$ for the loss of the resolved positive crossing, $-1$ for the loss of a 1-resolution, and $-3(2n_h-1)$ for the positive crossings becoming negative ($-1$ each for losing a positive crossing, and $-2$ each for adding a negative crossing).  We also have a homological shift of $-2n_h$ given the loss of a 1-resolution and the addition of $2n_h-1$ negative crossings.  The diagrams also make it clear that crossings away from this area retain their sign, so that these shifts are the only shifts present and we see:
\begin{equation}
\label{Hog cofib 1}
\X^{j-1+k\sum_{i=1}^\ell \tau_i} (D_0) \hookrightarrow
\X^{j+k\sum_{i=1}^\ell \tau_i} (D_{L_\gamma^{o(h)}}(k)) \twoheadrightarrow
\Sigma^{-2n_h}\X^{j+1-6n_h + k\sum_{i=1}^\ell \tau_i} (D_1).
\end{equation}

At this point, we first use an isotopy (Reidemeister moves) to rearrange $D_0$ and $D_1$ into $D'_0$ and $D'_1$ respectively (also shown in the diagram).  The $D'_0$ and $D'_1$ are then diagrams with braids above the $\T_{n_h}^k$.  The shifts in Equation (\ref{Hogancamp cofib}) are obtained from those in Equation (\ref{Hog cofib 1}) by straightening these braids (all positive crossings for $D'_0$, and all negative for $D'_1$) as in Corollary \ref{X straighten braids}.
\end{proof}

\begin{remark}
\label{hog cofib rmk}
There are similar cofibration sequences for a $(-1)$-linking unknot (ie switching the orientation of the unknot $\alpha_{\ell+1}$ in Figures \ref{Hopflike linking} and \ref{Hopflike resolutions}, and tracking the resulting $q$-degree shifts).  We can also see that in this proof, the choice of right- or left-handed twists is irrelevant, since none of the crossings or orientations of the strands passing through the twists were affected at any point.
\end{remark}

\section{A Tail for the Colored Khovanov Homotopy Type of B-Adequate Links}

\subsection{Discussion and Strategy}

This section is dedicated to proving Theorem \ref{B-adequate tail}.  Before investigating the details of the proof, we outline the general strategy and logic, expanding on the summary given in the introduction.  The goal is to adapt Rozansky's proof in \cite{Roz2} of the fact that the colored Khovanov homology groups of B-adequate links stabilize as the color goes to infinity.  The proof in that paper builds maps $f_n$ that give isomorphisms of colored homology groups between the $n$-colored and $(n+1)$-colored link $L$, but only within a certain homological range.  In order to prove Theorem \ref{B-adequate tail} then, it is enough to show that
\begin{itemize}
\item The maps $f_n$ in \cite{Roz2} are induced by maps $F_n$ between colored homotopy types, at least within the homological range of isomorphism.
\item If $n$ is large enough, the homological range of isomorphism guaranteed by Rozansky is enough to cover all non-zero homology of the corresponding colored homotopy types (and thus the $F_n$ induce isomorphisms on all homology, and so give stable homotopy equivalences by Whitehead's theorem).
\end{itemize}

Neither of these statements is difficult to prove conceptually, but the notation involved becomes somewhat cumbersome.  The reason is that, on the one hand, the colored homotopy type is a homotopy colimit, and in order to build maps we resort to finite approximations (ie the corresponding diagram with high twisting of the cables).  This requires $q$-degree shifts depending on $k$.  On the other hand, the maps $f_n$ built by Rozansky are compositions of a large number of simpler maps, many of which themselves shift the $q$-degree which will lead to separate $q$-degree shifts depending on $n$.  In addition, the maps were built with the use of the categorified Jones-Wenzl projectors rather than finite-twist approximations of them.  Thus some care will be needed.

Throughout this section, following \cite{Roz2}, all of the twisting will be \emph{left-handed} (ie, using $\X_-(D)$ from the proof of Theorem \ref{Projectors X exists}).  We recall here that, in addition to shifts of the form $k\tau$ for the normalization shift $n^+-2n^-$, the left-handed sequence also requires shifts of the form $kn(n-1)$ for counting the total number of crossings within the twist, accounting for 1-resolutions needed to move backward in the sequence.  See Equation (\ref{neg seq}).

\subsection{Notation and a Restatement of Theorem \ref{B-adequate tail}}

We begin with some notation.  Some of this is repeated from previous sections but is recalled here for convenience.

\begin{itemize}
\item $L$ denotes an oriented B-adequate link diagram, with kinks added as necessary to allow for the blackboard framing to be used.
\item $\chi$ will denote the total number of crossings in $L$.
\item $\pi$ will denote the total number of positive crossings in $L$ (only important for the homological shift, which will be ignored as often as possible).
\item $\chi^!$ will denote the total number of crossings in a minimal diagram for $L$ (only important for one key bound).
\item $\zeta$ will denote the total number of circles present in the all-1 resolution of $L$.
\item $\X_c^j(L_n)$ will denote the colored Khovanov homotopy type, in $q$-degree $j$, of the link $L$ with all of its components colored with the natural number $n$; see Definition \ref{Colored X def}.
\item For each $(n,k)\in\N^2$, $L(n,k)$ will denote the diagram obtained from $L$ by cabling all components with $n$ parallel strands, and adding a twist of $\T^{-k}_n$ to each cabling between \emph{every crossing}.  That is, if we replace $L$ with the graph with vertices at crossings and edges for strands between them, then each edge would be assigned a $\T^{-k}_n$ (see the beginning of section 4 in \cite{Roz2}).
\item $m$ will denote the total number of twistings $\T_n^{-k}$ coming from Jones-Wenzl projectors in the diagram $L(n,k)$.  This plays a similar role to $\ell$, the number of components of the link $L$, in the previous section.  However, as the previous item suggests, $m>\ell$ for our diagrams since we will be placing many such twistings on each component.
\item For any oriented diagram (link or tangle) $D$, $N_D$ will denote the normalization shift $n^+-2n^-$ counting all crossings in $D$.
\end{itemize}

The following notation is important enough to warrant its own definition.
\begin{definition}
\label{snk}
For a given $L$, the \emph{Colored $q$-degree Shift} is the integer function $s(n,k)$ that counts the normalization shift, the number of crossings, and the number of circles in the all-1 resolution of the link $L(n,k)$.  That is, with notation as above,
\begin{equation}
\label{s(n,k) definition}
s(n,k):= N_{L(n,k)}+kmn(n-1)+n^2\chi+n\zeta.
\end{equation}
\end{definition}
\begin{remark}
\label{nzeta remark}
Note that $n\zeta$ is the proper count for the number of circles in the all-1 resolution of $L(n,k)$, since the $\T_n^{-k}$'s present will become $\I_n$'s, and the all-1 resolution of a crossing coming from the original diagram gives a cabled version of the same resolution as in Figure \ref{allonecrossing}.
\end{remark}

\begin{figure}[h]
\centering
\vspace{.1in}\includegraphics[scale=.4]{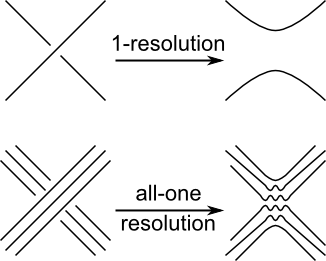}
\caption{Illustration of the all-1 resolution of a crossing in a cabled diagram}
\vspace{.1in}
\label{allonecrossing}
\end{figure}

Before moving forward, we illustrate the use of these notations to restate the result of Theorem \ref{Colored X exists}:
\begin{theorem}[Theorem \ref{Colored X exists} Restated For Uni-Colored Links]
\label{new Colored X exists}
For any oriented, colored link $L_n$ with the coloring $n$ on every component, there exists a colored Khovanov homotopy type $\X_c(L_n):=\bigvee_{j\in\mathbb{Z}} \X^{j+s(n,0)}_c(L_n)$ with wedge summands defined to be the homotopy colimits of the following sequences
\begin{equation}
\label{new Colored X equation}
\X^{j+s(n,0)}(L(n,0)) \twoheadleftarrow
\cdots
\twoheadleftarrow
\X^{j+s(n,k)}(L(n,k)) \twoheadleftarrow
\cdots
\end{equation}
which stabilize for large enough $k$.  In particular, for large enough $k$ we have a finite-twist approximation for $\X^{j+s(n,0)}_c(L_n)$ as
\begin{equation}
\label{finite-twist approx}
\X_c^{j+s(n,0)}(L_n) \simeq \X^{j+s(n,k)}(L(n,k))
\end{equation}
\end{theorem}
\begin{remark}
The term $s(n,0)$ is included in the original wedge summand for $\X_c(L_n)$ for convenience moving forward; note that the terms $n^2\chi$ and $n\zeta$ in Equation (\ref{s(n,k) definition}) are independent of $k$, and simply persist throughout the sequence (\ref{new Colored X equation}).
\end{remark}
\begin{proof}
This is essentially the sequence built in the proof of Theorem \ref{Projectors X exists} for $\X_-(D)$ as applied to Theorem \ref{Colored X exists}, except that extra projectors (and thus extra copies of $\T_n^{-k}$) are present.  These extra projectors cause no issues, however, thanks to Corollary \ref{X idempotent}.  In the proof of Theorem \ref{Projectors X exists}, the shift in the sequence includes a normalization term $-\sum k_i\tau_i$ and a crossing counting term $\sum k_in_i(n_i-1)$.  Here the $k_i$ and $n_i$ are all equal, and both terms are then absorbed into the shift $s(n,k)$.  Meanwhile, the left-handed twisting of a cabling where all strands are oriented the same way (in accordance with the orientation of $L$) means that all of the crossings involved are negative.  This ensures that the homological shifts cancel out (we lose negative crossings at the same rate that we lose 1-resolutions), so no suspensions are necessary.
\end{proof}

We now restate Theorem \ref{B-adequate tail} in a more precise fashion.

\begin{theorem}
\label{B-adequate tail better}
Fix an oriented B-adequate link diagram $L$.  With notation as above, there exist sequences of maps for each $j\in\mathbb{Z}$
\begin{equation}
\label{B-adequate grand sequence}
\X^{j+s(1,0)}_c(L_1) \twoheadleftarrow
\Sigma^{3\pi}\X^{j+s(2,0)}_c(L_2) \twoheadleftarrow
\cdots
\Sigma^{(n^2-1)\pi}\X^{j+s(n,0)}_c(L_n) \twoheadleftarrow
\cdots
\end{equation}
that become stable homotopy equivalences for $n>\chi^!-2j+1$.
\end{theorem}

This version of Theorem \ref{B-adequate tail} is the desired final result.  However, as indicated in the previous section, we actually build the required maps by taking finite-twist approximations for the various $\X_c(L_n)$.  With the help of Equation (\ref{finite-twist approx}) we translate Theorem \ref{B-adequate tail better} into the following:

\begin{theorem}
\label{B-adequate tail with twists}
Fix an oriented B-adequate link diagram $L$.  With notation as above, for each $j\in\mathbb{Z}$ and for each $(n,k)\in\N^2$, there exists a map denoted $F_{n,k,j}$ as shown below:
\begin{equation}
\label{Fnk definition}
F_{n,k,j}:\Sigma^{((n+1)^2-1)\pi}\X^{j+s(n+1,k)}(L(n+1,k))\twoheadrightarrow
\Sigma^{(n^2-1)\pi}\X^{j+s(n,k)}(L(n,k))
\end{equation}
such that, for large enough $k$, the following properties both hold:
\begin{enumerate}
\item Both the $\X(L(n,k))$ and $\X(L(n+1,k))$ terms are stably homotopy equivalent to their respective colored Khovanov homotopy types, so that $F_{n,k,j}$ provides the map $F_{n,j}$ below:
\begin{equation}
\label{Fn definition}
\begin{gathered}
\Sigma^{((n+1)^2-1)\pi}\X^{j+s(n+1,0)}_c(L_{n+1})\\
\vsimeq\\
\Sigma^{((n+1)^2-1)\pi}\X^{j+s(n+1,k)}(L(n+1,k))\\
\twoheaddownarrow\\
\Sigma^{(n^2-1)\pi}\X^{j+s(n,k)}(L(n,k))\\
\vsimeq\\
\Sigma^{(n^2-1)\pi}\X^{j+s(n,0)}_c(L_n)
\end{gathered}
\end{equation}
which is used to construct the sequence (\ref{B-adequate grand sequence}).
\item For $n>\chi^!-2j+1$, the map $F_{n,k,j}$ (and thus, $F_{n,j}$) is a stable homotopy equivalence.
\end{enumerate}
\end{theorem}

Before discussing the proof of this theorem, we provide a table and example to illustrate the statement of Theorem \ref{B-adequate tail better}.  The following lemma and corollary are provided to avoid useless clutter.

\begin{lemma}
\label{max j for Colored table}
For any link $L$, and for any $n\in\N$, we have that $j=0$ gives the maximal possible $q$-degree for non-zero colored homotopy type $\X_c^{j+s(n,0)}(L_n)$.
\end{lemma}
\begin{proof}
By the finite-twist approximation (\ref{finite-twist approx}), we have that
\[\X_c^{j+s(n,0)}(L_n) \simeq \X^{j+s(n,k)}(L(n,k))\]
for some large enough $k$.  The link $L(n,k)$ has Khovanov chain complex generator $z$ with maximal possible $q$-degree occurring in the all-one resolution, assigning $v^+$ to all of the circles.  Since the all-one resolution of the left-handed twists give identity braids, this generator $z$ has $q$-degree equal to:
\begin{align*}
\deg_q(z) &= \#(\text{1-resolutions}) + (\#(v_+) - \#(v_-)) + (n^+-2n^-)\\
&= \#(\text{crossings}) + \#(\text{circles}) + N_{L(n,k)}\\
&= s(n,k).
\end{align*}
which corresponds to $j=0$.
\end{proof}

\begin{corollary}
\label{no odd j for Colored table}
For any link $L$, and for any $n\in\N$, we have that $\X_c^{j+s(n,0)}(L_n)$ is trivial for odd $j$.
\end{corollary}
\begin{proof}
We see from the proof of Lemma \ref{max j for Colored table} that in any finite approximation for $\X_c^{j+s(n,0)}(L_n)$, there is a generator in $q$-degree corresponding to $j=0$.  The parity of $q$-degree is constant throughout the Khovanov chain complex, so we must have $j$ even.
\end{proof}

\begin{remark}
\label{meaning of j and s}
Lemma \ref{max j for Colored table} can be regarded as giving an alternative meaning for what the grading $j$, and the shift $s(n,k)$, are describing.  We see that $s(n,k)$ is precisely the maximum possible $q$-grading for the Khovanov chain complex of $L(n,k)$, and then $j$ is a measure of how far from that maximum we are.  This means $j\leq 0$, which correctly corresponds to building a power series in $q^{-1}$ for the rational terms in the decategorified setting of the projectors.
\end{remark}

We now present the general table of colored homotopy types for any link $L$ arranged to take advantage of Theorem \ref{B-adequate tail better}.

\begin{table}[h]
\vspace{.1in}
\begin{equation*}
\begin{array}{c|ccccccccc}
& j=0 && j=-2 && j=-4 && j=-6 && \dots\\
\hline
\X_c(L_1) & \X_c^{s(1,0)}(L_1) & \vee & \X_c^{s(1,0)-2}(L_1) & \vee & \X_c^{s(1,0)-4}(L_1) & \vee & \X_c^{s(1,0)-6}(L_1) & \vee & \cdots\\
& \\
\X_c(L_2) & \X_c^{s(2,0)}(L_2) & \vee & \X_c^{s(2,0)-2}(L_2) & \vee & \X_c^{s(2,0)-4}(L_2) & \vee & \X_c^{s(2,0)-6}(L_2) & \vee & \cdots\\
& \\
\X_c(L_3) & \X_c^{s(3,0)}(L_3) & \vee & \X_c^{s(3,0)-2}(L_3) & \vee & \X_c^{s(3,0)-4}(L_3) & \vee & \X_c^{s(3,0)-6}(L_3) & \vee & \cdots\\
\vdots & \vdots && \vdots && \vdots && \vdots
\end{array}
\end{equation*}
\caption{The table of uni-colored Khovanov homotopy types for a link $L$, with the vertical axis indicating color via subscript on $L$ and the horizontal axis indicating the suitably normalized $q$-degree; stabilization occurs vertically starting at a color that depends on both $j$ (the column) and $L$.}
\vspace{.1in}
\label{Colored table}
\end{table}

With Table \ref{Colored table} in mind, we can reinterpret some of the theorems stated above.

\begin{itemize}
\item Theorem \ref{new Colored X exists} guarantees that all of the colored Khovanov homotopy types in Table \ref{Colored table} exist, and Equation \ref{finite-twist approx} guarantees that any one of them is stably homotopy equivalent to the homotopy type of a finite-twist approximation $L(n,k)$.  Note that there is no single bound for $k$ that approximates all of the homotopy types in the table at once, since the bound would depend on both $j$ and $n$.
\item Theorem \ref{B-adequate tail better} asserts that there are `vertical' maps connecting all of the terms in any column of Table \ref{Colored table}, and furthermore that these maps are stable homotopy equivalences for $n>\chi^!-2j+1$.  Thus in any given column (fixed $j$) we see that the homotopy types are all stably equivalent for large enough $n$.  This is the general statement of Theorem \ref{B-adequate tail}.
\item Theorem \ref{B-adequate tail with twists} is the stepping stone to proving Theorem \ref{B-adequate tail better}.  It asserts the existence of the vertical maps \emph{after} replacing each entry in Table \ref{Colored table} by its corresponding finite-twist approximation as guaranteed by Equation \ref{finite-twist approx}.  Since we build the maps one at a time, we can focus on two adjacent entries in one column of the table (fix $j$ and focus on $n$ and $n+1$ for some $n$) and take $k$ to be larger than both stability bounds for these two entries.  Then this vertical map composes with the finite-twist approximation equivalences as in Equation (\ref{Fn definition}) to give the maps asserted by Theorem \ref{B-adequate tail better}.
\end{itemize}

To illustrate the stabilization as $n\rightarrow\infty$, we build the table for $L$ being the simplest non-trivial link, that is, the positive Hopf link.

\begin{example}
\label{Hopf link example}
Let $L$ be the positive Hopf link.  The reader can quickly verify that
\begin{gather*}
\chi=\chi^!=2\\
\zeta=2\\
N_L=2\\
\begin{aligned}
s(n,0) &= N_{L(n,0)}+0+n^2\chi+n\zeta\\
&= n^2N_L + n^2\chi + n\zeta\\
&= 4n^2+2n
\end{aligned}
\end{gather*}
which means that the bound $n>\chi^!-2j+1$ for stabilization becomes
\[n>3-2j.\]
Thus we have Table \ref{Hopf link Colored table} for the Hopf link.

\begin{table}[h]
\vspace{.1in}
\begin{equation*}
\begin{array}{c|ccccccccc}
& j=0 && j=-2 && j=-4 && j=-6 && \dots\\
\hline
\X_c(L_1) & \X_c^6(L_1) & \vee & \X_c^4(L_1) & \vee & \X_c^2(L_1) & \vee & \X_c^0(L_1)\\
& \\
\X_c(L_2) & \X_c^{20}(L_2) & \vee & \X_c^{18}(L_2) & \vee & \X_c^{16}(L_2) & \vee & \X_c^{14}(L_2) & \vee & \cdots\\
& \\
\X_c(L_3) & \X_c^{42}(L_3) & \vee & \X_c^{40}(L_3) & \vee & \X_c^{38}(L_3) & \vee & \X_c^{36}(L_3) & \vee & \cdots\\
& \\
\X_c(L_4) & \X_c^{72}(L_4) & \vee & \X_c^{70}(L_4) & \vee & \X_c^{68}(L_4) & \vee & \X_c^{66}(L_4) & \vee & \cdots\\
& \vsimeq\\
\X_c(L_5) & \X_c^{110}(L_5) & \vee & \X_c^{108}(L_5) & \vee & \X_c^{106}(L_5) & \vee & \X_c^{104}(L_5) & \vee & \cdots\\
& \vsimeq\\
\X_c(L_6) & \X_c^{156}(L_6) & \vee & \X_c^{154}(L_6) & \vee & \X_c^{152}(L_6) & \vee & \X_c^{150}(L_6) & \vee & \cdots\\
& \vsimeq\\
\X_c(L_7) & \X_c^{210}(L_7) & \vee & \X_c^{208}(L_7) & \vee & \X_c^{206}(L_7) & \vee & \X_c^{204}(L_7) & \vee & \cdots\\
& \vsimeq\\
\X_c(L_8) & \X_c^{272}(L_8) & \vee & \X_c^{270}(L_8) & \vee & \X_c^{268}(L_8) & \vee & \X_c^{266}(L_8) & \vee & \cdots\\
& \vsimeq & & \vsimeq\\
\vdots & \vdots && \vdots && \vdots && \vdots
\end{array}
\end{equation*}
\caption{The table of uni-colored Khovanov homotopy types for the positive Hopf link $L$; the vertical stable homotopy equivalences begin when $n>3-2j$, illustrated in the first two columns.}
\vspace{.1in}
\label{Hopf link Colored table}
\end{table}
Notice that in the second column of Table \ref{Hopf link Colored table}, stabilization begins after $n=8$ ($n>3-2(-2)=7$).  Also, note the absence of horizontal dots in the first row.  When $n=1$, the colored Khovanov homology (and homotopy type) is just the usual Khovanov homology (and homotopy type), which we know only exists in these 4 $q$-degrees for the positive Hopf link $L$.
\end{example}

\subsection{The Proof}

As mentioned in the discussion on strategy above, the maps $F_{n,k,j}$ will be lifts of the maps $f_n$ defined in Theorem 2.12 of \cite{Roz2}.  In that paper, Rozansky considers these as grading-preserving maps between `shifted colored Khovanov homology groups':
\begin{equation}
\label{Roz fn basic}
f_n:\tilde{H}^{i_R,j_R}(L_n) \longrightarrow \tilde{H}^{i_R,j_R}(L_{n+1})
\end{equation}
where we have used $i_R$ and $j_R$ to denote Rozansky's grading conventions.  \cite{Roz2} asserts the existence of these maps, and the fact that they are isomorphisms so long as $i_R\leq n-1$.

Here, we first provide the translation between Rozansky's grading conventions and our own.  The reader can verify from \cite{Roz2} that
\begin{align}
\label{Roz gradings}
i_R =& \#\text{cros}-\#\text{1-resolutions} = \#\text{0-resolutions}\\
j_R =& -(\#(v_+) - \#(v_-))+n\zeta
\end{align}
where the $\#\text{cros}$ term refers to the total number of crossings in the diagram $D_{L_n}$ (see Figure \ref{Colored diagrams} in Section 4).  From this and Equations (\ref{homdeg}) and (\ref{qdeg}) we see that
\begin{align}
\label{Roz to my gradings}
i &= -i_R+n^+\\
j &= (-i_R-j_R)+(n^+-2n^-)+\#\text{cros}+n\zeta
\end{align}
where the $n^+$ and $n^-$ are counting positive and negative crossings in the diagram $D_{L_n}$.  Although some further simplifications are possible, this format most clearly matches the format seen in the sequence (\ref{B-adequate grand sequence}) of Theorem \ref{B-adequate tail better} involving the $s(n,k)$ shift.

Now these colored homology groups use the diagrams $D_{L_n}$ containing the categorified Jones-Wenzl projectors.  In \cite{Roz} these categorified projectors are defined as stable limits of complexes using $\T_n^{-k}$ in place of the projectors, as in the proof of Theorem \ref{Colored X exists}.  This means that for large enough $k$ the following homology groups match:
\begin{gather*}
\tilde{H}^{i_R,j_R}(L_n)\cong \tilde{H}^{i_R,j_R}(D_{L_n}(k))\\
\tilde{H}^{i_R,j_R}(L_{n+1}) \cong \tilde{H}^{i_R,j_R}(D_{L_{n+1}}(k))
\end{gather*}
so long as $i_R\leq n-1$, the homological range which we are interested in.  Thus we may focus on these finite-twist approximations of the colored links $L_n$, and the maps $f_n$ in this context will give rise to the maps $F_{n,k,j}$ we seek.  The reader may check that the grading shifts now correspond to those present in Equation (\ref{Fnk definition}).

Now we prove the two lemmas that correspond to the two points discussed in the beginning of this section.  For the first lemma, we avoid going into detail about the precise definition of the maps $f_n$; the interested reader should consult sections 3 and 4 of \cite{Roz2}.
\begin{lemma}
\label{lift Roz maps}
The maps $f_n$ of Rozansky can be lifted to maps $F_{n,k,j}$ as in Equation (\ref{Fnk definition}).
\end{lemma}
\begin{proof}
The maps $f_n$ are built out of several sorts of maps corresponding to local transformations as in Section 4 of \cite{Roz2}:
\begin{enumerate}
\item Reidemeister moves involving strands away from the projectors.
\item Short exact sequences of complexes arising from resolving a crossing away from the projectors.
\item `Straightening braids' via resolving crossings adjacent to projectors.
\item Adding new $P_n$ projectors adjacent to an existing $P_{n+1}$ projector, and other similar uses of the idempotent-like behavior of the categorified projectors.
\item `Sliding' projectors above and below other strands.
\item Viewing the categorified $P_{n+1}$ as a cone of a map $C\rightarrow\I_{n+1}$ where the complex $C$ involves no identity braid diagrams (there are further grading conditions; see both \cite{Roz} and \cite{Roz2}).  This allow a short exact sequence roughly of the form $KC\left( \la C , \Z \ra \right) \hookrightarrow KC\left( \la P_{n+1} , \Z \ra \right) \twoheadrightarrow KC\left( \la \I_{n+1} , \Z \ra \right)$.
\end{enumerate}

The first two types of maps clearly extend first to the finite-twist approximations, then to the corresponding homotopy types (type 1 can be viewed as the content of Section 6 of \cite{LS}, while type 2 is Lemma \ref{Lemma3.32} also based on \cite{LS}).  Types 3 and 4 lift in a manner corresponding to the proofs of Propositions \ref{X straighten braids} and \ref{X idempotent} respectively, giving stable homotopy equivalences for large enough $k$.  Type 5 is just a combination of Reidemeister moves on the level of the finite-twist approximation, as in the proof of well-definedness of the colored homotopy type (proof of Theorem \ref{Colored X exists} in Section 4.2).

For type 6, we return to \cite{Roz} where the cone format of the categorified $P_{n+1}$ is derived based on the finite-twist approximations, which exhibit this cone structure via resolving all of the crossings in the twisting.  And so this map lifts to the homotopy type as a long composition of maps coming from the cofibrations (\ref{Lemma3.32onX}) which, on the level of homology, is precisely the desired map.

We note here that some of these maps giving stable homotopy equivalences (especially types 3 and 4) rely not just on Rozansky's bounds, but in the new setting on a proper lower bound for $k$.  Since there are only finitely many such moves used to build the map $f_n$, we can always force $k$ to be large enough to satisfy all of these lower bounds before we begin.
\end{proof}

The second lemma requires the following theorem from \cite{Roz2}.
\begin{theorem}[\cite{Roz2} Theorem 2.1]
\label{jR bound}
Using the notation of Equation (\ref{Roz fn basic}), we have that $\tilde{H}^{i_R,j_R}(L_n)=0$ for $j_R<-\frac{1}{2}(i_R+\chi^!)$.
\end{theorem}
\begin{proof}
This is one of several bounds on non-zero shifted colored Khovanov homology provided by Theorem 2.1 in \cite{Roz2}.  It is treated as a corollary of Theorem 2.11 which is proved with a spectral sequence built from the multicone presentation of the colored Khovanov chain complex resulting from resolving crossings away from the projectors.  See Section 5 of that paper.
\end{proof}

Using this result we can prove the following.
\begin{lemma}
\label{n bound}
Fix $j\in\mathbb{Z}$.  Then for $n>\chi^!-2j+1$, we have (for large enough $k$)
\[H^i(\Sigma^{(n^2-1)\pi}\X^{j+s(n,k)}(L(n,k)))=0\]
for all $i<n^2\pi-n+1$, which is equivalent to all $i_R>n-1$ for $\tilde{H}^{i_R,j_R}(L_n)$.
\end{lemma}
\begin{proof}
For large enough $k$ we have
\[H^i(\Sigma^{(n^2-1)\pi}\X^{j+s(n,k)}(L(n,k)))\cong
H^i(\Sigma^{(n^2-1)\pi}\X^{j+s(n,0)}_c(L_n))\cong
\tilde{H}^{i_R,j_R}(L_n)\]
Definition \ref{snk} describes $s(n,k)$ as a count of normalizations, crossings, and circles.  This allows us to use Equations (\ref{qdeg}) and (\ref{Roz gradings}) to convert:
\begin{align*}
j+s(n,k) = j+N_{L(n,k)}+\#\text{crossings}(L(n,k))+n\zeta &= \#(\text{1-resolutions}) + (\#(v_+) - \#(v_-)) + N_{L(n,k)}\\
j+\#\text{0-resolutions}+n\zeta &= (\#(v_+) - \#(v_-))
\end{align*}
so that
\begin{align*}
j_R=& -\#(v_+-v_-)+n\zeta\\
=& -j - \#\text{0-resolutions}\\
=& -j-i_R.
\end{align*}
The last line follows from the fact that the suspensions are designed to ensure that $i_R$ counting 0-resolutions in $L_n$ is the same as counting 0-resolutions in the finite-twist approximation $L(n,k)$.  A similar (and simpler) conversion ensures that the bound $i<n^2\pi-n+1$ is equivalent to $i_R>n-1$.  Meanwhile, the bound $n>\chi^!-2j+1$ quickly yields 
\[j>\frac{1}{2}(\chi^!-n+1).\]
Combining all of these gives, for $n>\chi^!-2j+1$ and $i<n^2\pi-n+1$ ($i_R>n-1$),
\begin{align*}
j_R &= -j-i_R\\
&< -\frac{1}{2}(\chi^!-n+1)-i_R\\
&< -\frac{1}{2}(\chi^!+i_R)
\end{align*}
which is precisely the bound of Theorem \ref{jR bound} for zero homology as desired.
\end{proof}

\begin{proof}[Proof of Theorem \ref{B-adequate tail with twists}]
From Lemma \ref{lift Roz maps}, we have the existence of the required maps $F_{n,k,j}$ that induce isomorphisms on homology for all homological gradings corresponding to $i_R\leq n-1$.  From Lemma \ref{n bound}, once $n>\chi^!-2j+1$ all of the spaces involved have zero homology in all homological gradings corresponding to $i_R>n-1$.  Therefore for $n>\chi^!-2j+1$ the maps $F_{n,k,j}$ induce isomorphisms on all homology groups, and so by Whitehead's theorem they are stable homotopy equivalences as desired.
\end{proof}

As noted above, this provides the proof of Theorem \ref{B-adequate tail}.

\section{A More Explicit Tail for the Colored Khovanov Homotopy Type of the Unknot}

\subsection{The Idea}

In this final section we prove Theorem \ref{n to infty} by giving an alternative, more explicit proof showing the tail behavior for the colored Khovanov homotopy type of the unknot.  Since cabling an unknot with a torus braid twist simply produces the torus links $T(n,m)$, we use the notation $\X(T(n,\infty))$ for the homotopy type of the $n$-colored unknot.
\begin{remark}
\label{Tnm not Tnk}
There is an important distinction to be made here.  Earlier, the notation $\T_n^k$ was used to denote a torus \emph{braid} consisting of $k$ \emph{full (right-handed) twists}.  Now we use the notation $T(n,m)$ to denote a torus \emph{link} with $m$ \emph{fractional $\frac{1}{n}^{\text{th}}$ (right-handed) twists} as in \cite{MW}.
\end{remark}

Before going into the details of the proof, we provide a table to illustrate the goal of the construction, similar to Table \ref{Colored table} for the general case.  First, a quick lemma that can be regarded as the translation of Lemma \ref{max j for Colored table} and Corollary \ref{no odd j for Colored table} into this setting, presented to avoid needless clutter.
\begin{lemma}
\label{parity}
For $n \not\equiv j$ mod 2, $\X^j(T(n,\infty))$ is trivial.  Likewise, for $j<-n$, $\X^j(T(n,\infty))$ is trivial.
\end{lemma}
\begin{proof}
This is a simple consequence of Corollaries 7.2, 7.3, and 7.4 in \cite{MW}.  For $n=2$ and $n=3$, the statement is clear from the formulas presented there.  For $n\geq 4$, one can compute the minimal $q$-degree available in the all 0-resolution of the relevant torus link (which is just $-n$ + the number of crossings), which gives a minimal $q$-degree available for $\X^j(T(n,\infty))$ (notice that the degree shift in the formulas of Corollary 7.2 is precisely the number of crossings in the relevant torus link).  Since the parity of $q$-degree is constant throughout the Khovanov chain complex of a given link, the parity of this $q$-degree can also be used to prove the first statement (after the indices are shifted properly).  See Corollary \ref{no odd j for Colored table}, or Lemma 7.6 in \cite{MW}, for a more detailed discussion of this idea.
\end{proof}

With Lemma \ref{parity} in hand, we can construct Table \ref{Unknot colored table} for the colored unknot.

\begin{table}[h]
\vspace{.1in}
\begin{equation*}
\begin{array}{c|ccccccccc}
& j=0 && j=2 && j=4 && j=6 && \dots\\
\hline
\X(T(1,\infty)) & \X^{-1}(T(1,\infty)) & \vee & \X^1(T(1,\infty))\\
& \vsimeq\\
\X(T(2,\infty)) & \X^{-2}(T(2,\infty)) & \vee & \X^0(T(2,\infty)) & \vee & \X^2(T(2,\infty)) & \vee & \X^4(T(2,\infty)) & \vee & \cdots\\
& \vsimeq && \vsimeq\\
\X(T(3,\infty)) & \X^{-3}(T(3,\infty)) & \vee & \X^{-1}(T(3,\infty)) & \vee & \X^1(T(3,\infty)) & \vee & \X^3(T(3,\infty)) & \vee & \cdots\\
& \vsimeq && \vsimeq\\
\X(T(4,\infty)) & \X^{-4}(T(4,\infty)) & \vee & \X^{-2}(T(4,\infty)) & \vee & \X^0(T(4,\infty)) & \vee & \X^2(T(4,\infty)) & \vee & \cdots\\
&  \vsimeq && \vsimeq && \vsimeq && \\
\vdots & \vdots && \vdots && \vdots && \vdots
\end{array}
\end{equation*}
\caption{The table of colored Khovanov homotopy types for the unknots, notated as homotopy types of torus links, with the horizontal axis indicating suitably normalized $q$-degree; the vertical stabilizations in each column besides the first begin at $\X^0$, which corresponds to color $n=j$.}
\vspace{.1in}
\label{Unknot colored table}
\end{table}

The goal of this section will be to construct the `vertical' stable homotopy equivalences already presented in Table \ref{Unknot colored table}.  Note that, like in the general case (Table \ref{Colored table}), the $j$ terms are arranged to `start' at zero, but now \emph{increase} in the positive direction.  This stems from the fact that we will be using right-handed twists rather than the left-handed twists considered in the previous section.  Also, since $T(1,\infty)$ is just an unknot, there is no need for an infinite wedge sum in the first row (similar to the first row in Table \ref{Hopf link Colored table} for the Hopf link; see Example \ref{Hopf link example}).

The construction of these vertical maps follows a simple observation.  It is well known that the torus links satisfy $T(n,n+1)\cong T(n+1,n)$.  The sequences used to build $\X(T(n,\infty))$ in \cite{MW} were based on going from $\X(T(n,m))\hookrightarrow\X(T(n,m+1))$.  We can combine these two ideas to see a `diagonal' sequence of the form (omitting the $T$ from the notation):

\begin{equation}
\label{n to infty idea}
\begin{array}{ccccccccccc}
\X(n,n-1) & \hookrightarrow & \X(n,n) & \hookrightarrow & \X(n,n+1) & \hookrightarrow & \cdots\\
& & & & \vsimeq \\
& & & & \X(n+1,n) & \hookrightarrow & \X(n+1,n+1) & \hookrightarrow & \X(n+1,n+2) & \hookrightarrow & \cdots\\
& & & & & & & & \vsimeq\\
& & & & & & & & \X(n+2,n+1) & \hookrightarrow & \cdots\\
& & & & & & & & & \ddots
\end{array}
\end{equation}

If we can find a lower bound on $n$ so that all of these maps are stable homotopy equivalences, including the horizontal dots (indicating that in fact $\X(n,n-1)\simeq\X(n,\infty)$, and similarly for the other rows), we would have stable equivalences between homotopy types $\X(n,\infty)$ as $n\rightarrow\infty$ as desired.  Of course, this cannot be done once and for all; instead, it is done one (shifting) $q$-degree at a time.  The vertical equivalences in Table \ref{Unknot colored table} will be precisely the resulting maps.

\subsection{The Proof}

The techniques used in this section borrow as much from the results in \cite{MW} as from the ideas in this paper.  Thus we recall some notation from \cite{MW} below.

\begin{definition}
\label{old limiting def}
\item $\X(T(n,\infty))=\bigvee_{j\in\Z} \X^{j-n}(T(n,\infty))$, where for each $j\in\mathbb{Z}$,
\begin{equation}
\label{old limiting eqn}
\X^{j-n}(T(n,\infty)):= \text{hocolim}\left[\X^{j-n}(T(n,0))\hookrightarrow\cdots\hookrightarrow
\X^{(j-n)+m(n-1)}(T(n,m))\hookrightarrow\cdots\right].
\end{equation}
\end{definition}

This is the sequence making up the horizontal maps in the conceptual equation (\ref{n to infty idea}) above.  The original definition in \cite{MW} did not include the extra $q$-degree shift of $-n$ in the definition; this term has been included here for convenience, as suggested by the format of Table \ref{Unknot colored table}.  In fact this shift plays a role similar to that of the term $s(n,0)$ in Section 5.  (Indeed, since the unknot has no crossings and one circle in any resolution, this term is precisely $s(n,0)$; the negation is because we will be considering right-handed twisting rather than left-handed).

In \cite{MW} we prove that such sequences of maps become homotopy equivalences for large enough $m$.  We restate the precise result here, as we shall need a small improvement to the bound as well as a careful translation of the $q$-degrees to our new setting.

\begin{theorem}[Theorem 4.1 from \cite{MW}]
\label{old stab result}
Fix $a\in\Z$ and $n\in\N$.  Define $f(a,n):=\max(\frac{a+n-1}{n},n)$.  Then for any $m\geq f(a,n)$,
\begin{equation}
\label{old stab eq}
\X^a(T(n,m))\hookrightarrow\X^{a+n-1}(T(n,m+1))
\end{equation}
is a stable homotopy equivalence.
\end{theorem}

Our new improvement on the bound is very slight, but crucial.

\begin{lemma}
\label{new stab m=n-1}
The bound $f(a,n)$ in Theorem \ref{old stab result} can be improved to a new bound
\begin{equation}
\label{new stab bound}
f'(a,n):= \max(\frac{a+n-1}{n},n-1).
\end{equation}
\end{lemma}

The proof of Lemma \ref{new stab m=n-1} requires one small addition to the proof of Lemma 3.5 in \cite{MW}, which itself is a result of Sto\v{s}i\'{c}.  Since this will require re-introducing several notations from \cite{MW} that are not used elsewhere in this paper, we relegate this proof to an appendix.  Meanwhile, the following corollary translates the result above for use with the sequences (\ref{old limiting eqn}).

\begin{corollary}
\label{new stab bound for seq}
Fix $j\in\mathbb{Z}$ and $n\in\N$.  Then for $m\geq\max(j-1,n-1)$, the sequence (\ref{old limiting eqn}) stabilizes.  That is, the maps
\[\X^{(j-n)+m(n-1)}(T(n,m))\hookrightarrow\X^{(j-n)+(m+1)(n-1)}(T(n,m+1))\]
are stable homotopy equivalences.
\end{corollary}
\begin{proof}
The bound $m\geq n-1$ is the improvement (over $m\geq n$) of Lemma \ref{new stab m=n-1}.  Meanwhile, the term $\frac{a+n-1}{n}$ in the bound (\ref{new stab bound}) contains the $q$-degree $a$ which corresponds here to $(j-n)+m(n-1)$.  Some simple algebra ensures that
\[ m \geq \frac{a+n-1}{n} \hspace{.1in} \Longleftrightarrow \hspace{.1in} m\geq j-1.\]
\end{proof}

With these bounds in place, we are ready to provide the vertical equivalences of Table \ref{Unknot colored table} via the idea of Equation (\ref{n to infty idea}).

\begin{lemma}
\label{n to infty l1}
Fix $j\in\left(2\N\cup 0\right)$.  Then for $n\geq j$, we have
\begin{equation}
\label{n to infty eq1}
\begin{gathered}
\X^{j-n}(T(n,\infty)) \simeq \X^{j-n+(n-1)^2}(T(n,n-1))\\
\X^{j-(n+1)}(T(n+1,\infty)) \simeq \X^{j-(n+1)+n^2}(T(n+1,n))
\end{gathered}
\end{equation}
\end{lemma}
\begin{proof}
This follows directly from Corollary \ref{new stab bound for seq}.  When $n\geq j$, the $n-1$ term dominates in the bound $m\geq\max(j-1,n-1)$, allowing the sequence (\ref{old limiting eqn}) to stabilize as soon as $m=n-1$.  Of course, if $n\geq j$, then $n+1\geq j$ as well.
\end{proof}

\begin{lemma}
\label{phi lemma}
For $n\geq j$ as above, define the map $\phi_{n,j}$ to be the composition below.
\begin{gather*}
\X^{j-n+(n-1)^2}(T(n,n-1))\\
\downarrow\\
\X^{j-n+(n-1)^2+(n-1)}(T(n,n))\\
\downarrow\\
\X^{j-n+(n-1)^2+2(n-1)}(T(n,n+1))\\
\vsimeq\\
\X^{j-(n+1)+n^2}(T(n+1,n))
\end{gather*}
where the first two maps are the same maps appearing in the sequence (\ref{old limiting eqn}), and the final equivalence comes from the isotopy $T(n,m)\cong T(m,n)$.  Then $\phi_{n,j}$ defines a stable homotopy equivalence
\[\phi_{n,j}:\X^{j-n+(n-1)^2}(T(n,n-1)) \xrightarrow{\simeq} \X^{j-(n+1)+n^2} (T(n+1,n)).\]
\end{lemma}
\begin{proof}
As in the previous lemma, the first two maps are stable homotopy equivalences due to the bound in Corollary \ref{new stab bound for seq}.
\end{proof}

\begin{remark}
We see that this map $\phi_{n,j}$ plays a role similar to that of the $F_{n,j}$ of the previous section, but is much easier to define than the maps $f_n$ in \cite{Roz2} that lead to $F_{n,j}$.
\end{remark}

\begin{proof}[Proof of Theorem \ref{n to infty}]
Combining Lemmas \ref{n to infty l1} and \ref{phi lemma} gives stable homotopy equivalences
\begin{align*}
\X^{j-n}&(T(n,\infty))\\
\text{Lemma \ref{n to infty l1}}\hspace{.5in}&\vsimeq \\
\X^{j-n+(n-1)^2}&(T(n,n-1))\\
\text{Lemma \ref{phi lemma}}\hspace{.5in}&\vsimeq\\
\X^{j-(n+1)+n^2}&(T(n+1,n))\\
\text{Lemma \ref{n to infty l1}}\hspace{.5in}&\vsimeq \\
\X^{j-(n+1)}&(T(n+1,\infty))
\end{align*}
for arbitrary $n\geq j$, which gives all of the necessary stable homotopy equivalences as indicated in Table \ref{Unknot colored table}.  The calculations presented in Theorem \ref{n to infty} refer to the beginning of the stabilization, that is when $n=j$ so that we are considering $\X^0(T(j,\infty))$.  A further application of Corollary \ref{new stab bound for seq} shows that $\X^0(T(j,\infty))\simeq \X^{(j-1)^2}(T(j,j-1))$ so long as $j>0$, while the $j=0$ case stabilizes immediately (ie for $n=1$) giving the homotopy type of an unknot, which is known to be the sphere spectrum in $q$-degrees $\pm 1$.
\end{proof}

\begin{remark}
\label{n to infty unlink}
It is clear that a similar argument could be used to define $\X(U_\gamma)$ for an unlink $U$ allowing the colors on each component to tend to infinity.  We do not go through the calculation here.
\end{remark}

We conclude with a brief discussion on the differences between the new approach of this section and the general approach of the previous one.  One difference is that we use right-handed twisting in this new approach, but this is of no consequence and a left-handed version of the new approach could easily be derived.  The important difference is that, in the general case, the stable homotopy equivalences required are based on Rozansky's maps $f_n$ which are very complicated, requiring multiple properties of the categorified projectors (idempotency, straightening adjacent braids, a careful multi-cone presentation).  Even with no crossings (as in the unknot or unlink), the passage from cabling with $n$ strands to cabling with $n+1$ requires extra projectors and clever manipulations between them.  In our new approach for the unknot, the only maps required are those that already arise in the stable sequence (\ref{old limiting eqn}) based on resolving crossings, and maps derived from Reidemeister moves providing the isotopy between $T(n,n+1)$ and $T(n+1,n)$.  In fact this new approach views the tail of the colored Khovanov homotopy types of the unknot as a stabilization (one $q$-degree at a time) of the sequence $\X(T(n+1,n))$ as $n\rightarrow\infty$, rather than as a statement about categorified projectors and colored homotopy types in the usual sense.

The simple form of the maps used in this approach also gives an improvement on the bound on $n$ for stabilization.  In Rozansky's approach, the bound grows like $2j$, while here the bound grows like $j$.  Compare Table \ref{Unknot colored table} to Table \ref{Hopf link Colored table} to see the gap between beginning of stabilization for adjacent columns in the two cases.

\appendix
\section{Proof of Lemma \ref{new stab m=n-1}}
In \cite{MW} the bound $m\geq n$ needed for stabilization appears solely due to its presence in Lemma 3.5 in that paper, which itself is a rephrasing of Lemma 1 in \cite{Sto}.  A careful reading of \cite{MW} ensures that this bound is never used explicitly again.  Thus our goal in this section is to prove this lemma holds in the case $m=n-1$ as well.

For this we recall the notations of \cite{MW}.  When $m=n-1$, the lemma is concerned with resolving crossings from $T(n,n)$ in order to arrive at $T(n,n-1)$.  In this situation, we define
\begin{itemize}
\item $D_0:= T(n,n)$.
\item For $i=1,\dots,n-1$, $D_i$ and $E_i$ are the diagrams obtained by resolving the `top-most' crossing of $D_{i-1}$ as a 0-resolution and 1-resolution respectively.  Thus we have the cofibration sequences (see \cite{MW} for the degree shifts)
\[\X(D_i) \hookrightarrow \X(D_{i-1}) \twoheadrightarrow \X(E_i)\]
and $D_{n-1}=T(n,n-1)$.
\item $c_i$ denotes the number of negative crossings present in $E_i$.
\end{itemize}
With these notations in place, we wish to prove:
\begin{lemma}
\label{new ci bound}
For all $i=1,\dots,n-1$,
\[c_i = 2n-3.\]
\end{lemma}
Proving this will verify the bound $c_i\geq n+m-2$ of Lemma 3.5 of \cite{MW} in the case $m=n-1$, and then the rest of the results of \cite{MW} go through to prove Lemma \ref{new stab m=n-1} here.

\begin{remark}
Note that these definitions for $D_i$ and $E_i$ are different from those used in the proof of Proposition \ref{main seq stab}, where arbitrary tangles are being considered away from the twisting and we `slide' the topmost twist over to be adjacent to the tangle rather than to the other twists before resolving crossings.  The pictures used in the proof below will make the difference clear, and will resemble the similar pictures in \cite{MW}.
\end{remark}

\begin{proof}[Proof of Lemma \ref{new ci bound}]
The case $c_1$ is considered separately from $c_{i>1}$.  For $E_1$ we see the diagram illustrated in Figure \ref{E1n-1}, where the strands are closed up outside of the picture in the usual way.  The red circles clearly indicate that the turnback can be pulled through via $n-2$ Reidemeister 2 moves, then a negative Reidemeister 1 move, then another $n-2$ Reidemeister 2 moves.  Each Reidemeister 2 move involves precisely one negative crossing, which quickly proves the claim.

\begin{figure}[h]
\centering
\vspace{.1in}\includegraphics[scale=.3]{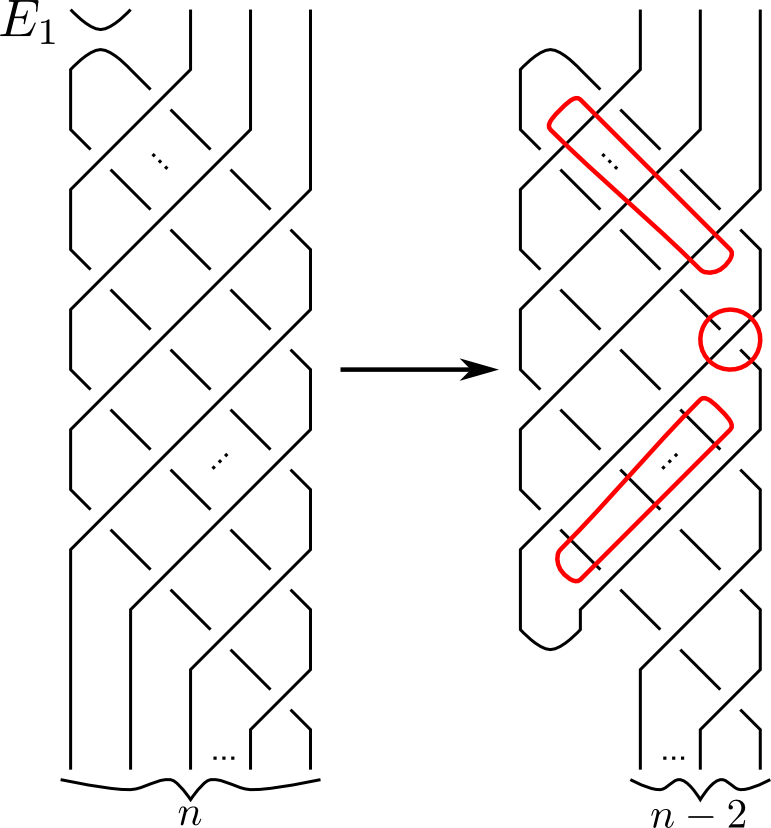}
\caption{The picture for $E_1$, where the topmost turnback is pulled around the cabling allowing for the diagram on the right; the red circles indicate Reidemeister moves that will occur while pulling the turnback through the twisting}
\vspace{.1in}
\label{E1n-1}
\end{figure}

For $E_{i>1}$, we see that the turnback can be pulled through the torus braid $\T(n,n-1)$ leaving us with a copy of $\T(n-2,n-3)$ as in Figure \ref{Ein-1} (note that we use the notation $\T$ to indicate the torus braid rather than the complete torus link; however, we continue to use the parentheses notation to indicate fractional twists rather than the full twists indicated by superscripts throughout the rest of the paper).
\begin{figure}[h]
\centering
\vspace{.1in}\includegraphics[scale=.3]{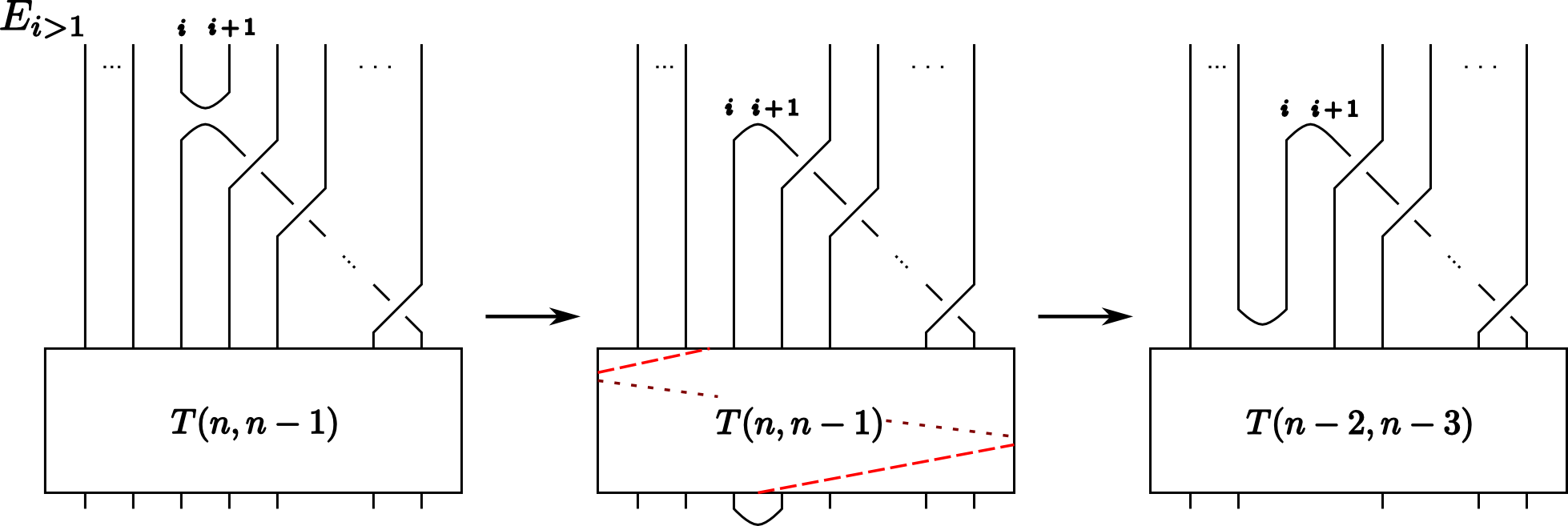}
\caption{The picture for $E_{i>1}$; the topmost turnback is pulled around the cabling and then through the torus braid $\T(n,n-1)$ along the indicated dashed line, eliminating two strands from the braid but keeping the twisting at one strand less than a full twist, leaving us with $\T(n-2,n-3)$}
\vspace{.1in}
\label{Ein-1}
\end{figure}
Now we count crossings similarly to the proof of Lemma \ref{Counting lemma}.  The initial braid $\T(n,n-1)$ had $(n-1)^2$ crossings, while the new $\T(n-2,n-3)$ has $(n-3)^2$ crossings.  The crossings `above' the braid remain unchanged, so the total change in the number of crossings is $(n-1)^2-(n-3)^2=4n-8$.  Since the turnback was able to swing completely around the entire torus braid (see the red dashed line in Figure \ref{Ein-1}), it must have accomplished precisely two (negative) Reidemeister 1 moves.  This leaves $4n-10$ crossings eliminated by Reidemeister 2 moves (see Figure \ref{E1n-1} to see that these are the only moves involved).  Thus half of the $4n-10$ crossings were negative, plus the two Reidemeister 1 moves gives precisely $2n-3$ as required.
\end{proof}




\bibliographystyle{alpha}

\bibliography{A_Colored_Khovanov_Homotopy_Type_And_Its_Tail_For_B_Adequate_Links}

\end{document}